\documentclass[11pt,leqno]{amsart}

\usepackage{graphicx}
\usepackage{color}
\usepackage{geometry}
\usepackage{fullpage}
\usepackage{calrsfs}
\usepackage[OT2,T1]{fontenc}
\usepackage[english]{babel}
\usepackage{times}

\usepackage{etoolbox}
\patchcmd{\section}{\scshape}{\bfseries}{}{}
\makeatletter
\renewcommand{\@secnumfont}{\bfseries}
\makeatother

\newtheorem{introtheorem}{Theorem}

\theoremstyle{definition}

\newtheorem*{introexample*}{Example}
\newtheorem*{introexamples*}{Examples}

\newtheorem*{introremark*}{Remark}
\newtheorem*{introremarks*}{Remarks}
\theoremstyle{plain}
\newtheorem{theorem}{Theorem}[section]
\newtheorem{proposition}[theorem]{Proposition}
\newtheorem{lemma}[theorem]{Lemma}

\theoremstyle{definition}
\newtheorem{definition}[theorem]{Definition}
\newtheorem{notation}[theorem]{Notation}
\newtheorem{example}[theorem]{Example}

\newtheorem{remark}[theorem]{Remark}
\newtheorem{condition}[theorem]{Condition}

\NewCommandCopy{\proofqedsymbol}{\qedsymbol}% save the default

\AtBeginEnvironment{proof}{\renewcommand{\qedsymbol}{\proofqedsymbol}}

\AtBeginEnvironment{introexample}{%
  \pushQED{\qed}\renewcommand{\qedsymbol}{$\lozenge$}%
}
\AtEndEnvironment{introexample}{\popQED}

\AtBeginEnvironment{introexample*}{%
  \pushQED{\qed}\renewcommand{\qedsymbol}{$\lozenge$}%
}
\AtEndEnvironment{introexample*}{\popQED}

\AtBeginEnvironment{example}{%
  \pushQED{\qed}\renewcommand{\qedsymbol}{$\lozenge$}%
}

\AtBeginEnvironment{introexamples*}{%
  \pushQED{\qed}\renewcommand{\qedsymbol}{$\lozenge$}%
}
\AtEndEnvironment{introexamples*}{\popQED}

\AtBeginEnvironment{example}{%
  \pushQED{\qed}\renewcommand{\qedsymbol}{$\lozenge$}%
}

\AtEndEnvironment{example}{\popQED}

\AtBeginEnvironment{examples}{%
  \pushQED{\qed}\renewcommand{\qedsymbol}{$\lozenge$}%
}
\AtEndEnvironment{examples}{\popQED}

\AtBeginEnvironment{introremark}{%
  \pushQED{\qed}\renewcommand{\qedsymbol}{$\lozenge$}%
}
\AtEndEnvironment{introremark}{\popQED}

\AtBeginEnvironment{introremarks*}{%
  \pushQED{\qed}\renewcommand{\qedsymbol}{$\lozenge$}%
}
\AtEndEnvironment{introremarks*}{\popQED}

\AtBeginEnvironment{remark}{%
  \pushQED{\qed}\renewcommand{\qedsymbol}{$\lozenge$}%
}
\AtEndEnvironment{remark}{\popQED}

\AtBeginEnvironment{remarks}{%
  \pushQED{\qed}\renewcommand{\qedsymbol}{$\lozenge$}%
}
\AtEndEnvironment{remarks}{\popQED}

\AtBeginEnvironment{definition}{%
  \pushQED{\qed}\renewcommand{\qedsymbol}{$\triangle$}%
}
\AtEndEnvironment{definition}{\popQED}

\AtBeginEnvironment{notation}{%
  \pushQED{\qed}\renewcommand{\qedsymbol}{$\triangle$}%
}
\AtEndEnvironment{notation}{\popQED}

\usepackage{amsfonts}
\usepackage{amssymb}
\usepackage{hyperref}
\usepackage{mathtools}
\usepackage{relsize}
\usepackage{aligned-overset}

\newcommand{\Z}{\mathbf{Z}}
\newcommand{\Q}{\mathbf{Q}}
\newcommand{\R}{\mathbf{R}}

\newcommand{\F}{\mathbf{F}}
\DeclareMathOperator{\End}{\mathrm{End}}

\newcommand{\Ga}{\mathbf{G}_a}
\newcommand{\Gm}{\mathbf{G}_m}
\renewcommand{\geq}{\geqslant}
\renewcommand{\leq}{\leqslant}

\makeatletter
\newcommand*{\defeq}{\mathrel{\rlap{%
                     \raisebox{0.24ex}{$\m@th\cdot$}}%
                     \raisebox{-0.24ex}{$\m@th\cdot$}}%
                     =}
\makeatother

\begin{document}

\date{\today}
\title{Orbit decomposition statistics for discrete dynamical systems:\\[1mm] the Ces\`aro mean and a large deviation principle}
\author[G.~Cornelissen]{Gunther Cornelissen}
\address{\normalfont Mathematisch Instituut, Universiteit Utrecht, Postbus 80.010, 3508 TA Utrecht, Nederland}
\email{g.cornelissen@uu.nl}
\author[S.W.~Park]{Sun Woo Park}
\address{\normalfont Max-Planck-Institut f\"ur Mathematik, Vivatsgasse 7, 53111 Bonn, Deutschland}
\email{s.park@mpim-bonn.mpg.de}
\thanks{The first author thanks the Hausdorff Institute of Mathematics (HIM) for hospitality during the month when this work was started, and both authors thank the Max-Planck-Institut f\"ur Mathematik (MPIM) for excellent working conditions. Through HIM, this research is funded by the Deutsche Forschungsgemeinschaft (DFG, German Research Foundation) under Germany‘s Excellence Strategy – EXC-2047/1 – 390685813. The second author also thanks the Centre de Recherches Math\'ematiques (CRM) in Montr\'eal for support during the thematic programme `Universal Statistics in Number Theory'. SageMath was used for computing examples in order to scrutinise some of the assymptotic results. AI (Anthropic’s Claude Sonnet 4.6 Adaptive) was used in the conceptualisation phase of some of the proofs, that were then written manually by the authors. We thank Marc Houben, Berend Ringeling, Tom Ward and Wadim Zudilin for helpful comments.}


\begin{abstract} \noindent If a self-map $\sigma \colon \mathcal{X} \rightarrow \mathcal{X}$ has a dynamical zeta function with nonzero radius of convergence $1/\Lambda$ and the Ces\`aro mean $B$ of $ \# \mathrm{Fix}(\sigma^k)/\Lambda^k$ exists and is positive, we show a large deviation principle for the number of prime orbits occurring in the decomposition of a general orbit of length $\leq X$ (an element of the free abelian monoid generated by the prime orbits or, equivalently, a prime orbit of a finite multiset in $\mathcal{X}$) with speed $B \log X$ and universal rate function equal to that of the Poisson distribution with unit mean. We also show a  large deviation principle for more general strongly additive functions. The proof uses asymptotic results on the total number of general orbits, as well as a weak analogue of Mertens's second theorem, that may be of independent interest. The theory applies, for example, to endomorphisms of algebraic groups over finite fields, additive cellular automata, and automorphisms of some solenoids. 
\end{abstract}

\maketitle

\section{Introduction} 
We consider a discrete dynamical system $(\mathcal{X},\sigma)$ given by a set $\mathcal{X}$ and iteration of a self-map $\sigma \colon \mathcal{X} \rightarrow \mathcal{X}$. Just like the positive integers under multiplication are the free multiplicative abelian monoid generated by the primes, prime  (i.e., finite closed) orbits of a point under iteration of $\sigma$ span a free additive monoid, whose elements, that we call \emph{general orbits}, correspond bijectively to orbits of finite multisets in $\mathcal{X}$; the cardinality of a prime orbit, its \emph{length}, extends additively to general orbits. Thus, it makes sense to study analogues for such dynamical systems of statistical properties of prime decomposition, such as the expected number of prime factors of an integer (Hardy--Ramanujan \cite{HardyRamanujan}, 1917), convergence to the normal distribution (Erd\H{o}s--Kac \cite{EK}, 1940), speed of convergence (R\'enyi--Tur\'an \cite{RenyiTuran}, 1958) and Large Deviation Principles (Mehrdad--Zhu \cite{MehrdadZhu}, 2016). 

Some such generalisations exist in the setting of normed abelian monoids with axiomatic assumptions on the corresponding prime counting function akin the prime number theorem (a research area known as `abstract analytic number theory' \cite{Knopfmacher} or the theory of `Beurling primes'  \cite{Beurling}), see, e.g., \cite{Liu1} for Erd\H{o}s--Kac and \cite{KP} for Large Deviations. This theory is not applicable to most of the discrete dynamical systems we will consider. 

Results like the Central Limit Theorem (of which the Erd\H{o}s--Kac theorem is a number-theoretical analogue) and Large Deviation Principles (LPD's for short, see, e.g.\ \cite{DemboZeitouni}) take place on different scales, and there is no reason to expect that their generalisations to other contexts should depend on the same assumptions or should be treated in historically chronological order. In this paper, we focus solely on deducing LPD's  for orbit decomposition in discrete dynamical systems, depending on minimalistic conditions: only the convergence properties of the (Artin--Mazur) dynamical zeta function and the existence of a Ces\`aro mean for the rescaled fixed point count need to be assumed. From this, an asymptotic count of general orbits, as well as a dynamical analogue of Mertens's second theorem from prime number theory can be deduced (other examples of such analogues can be found in \cite{BCav}, \cite{SharpMertens}). These results are then sufficient to apply the G\"artner--Ellis theorem and get a LDP. This typically answers questions of the form `what is the probability that a general orbit has a given fraction more than the expected number of components in its decomposition, depending on its length?'. The analogue of the Erd\H{o}s--Kac theorem seems to require stronger hypotheses, related to the error terms in the asymptotics, and we hope to return to this in a subsequent paper. Finally, note that for dynamical systems, one may also consider the statistics of orbit length itself instead of orbit decomposition, see, e.g., \cite{Kifer}, \cite{Sharp}; this is different from what we do here (our method only handles \emph{strongly} additive functions $g$, for which $g(P^k)=g(P)$ for prime orbits $P$).  

\subsection*{Notation} To describe our results in a precise way, we first fix some notation. Let $\sigma_k$ denote the number of fixed points of the $k$-th iterate $\sigma^k \colon \mathcal{X} \rightarrow \mathcal{X}$ of a self-map $\sigma \colon \mathcal{X} \rightarrow \mathcal{X}$ on a set $\mathcal{X}$. Assume that $\sigma$ is \emph{confined}, i.e., $\sigma_k$ is finite for all $k$. The \emph{growth rate} of $(\mathcal{X},\sigma)$ is \[ \Lambda \defeq \exp \limsup \log(\sigma_k)/k.\] 
A prime (or primitive, or finite) orbit for $\sigma$ of length $\ell$ is a set $O = O(x) = \{x, \sigma(x),\sigma^2(x),\dots,\sigma^{\ell}(x)=x\}$ of exact cardinality $\ell=\ell(O)$ for some $x \in \mathcal{X}$. A \emph{general orbit} is an element of the free additive abelian monoid generated by the prime orbits, with the length function extended additively. Equivalently, a general orbit is a finite orbit of a finite multiset in $\mathcal{X}$: if $S$ is a finite multiset in $\mathcal{X}$ consisting of elements $x_i$ with multiplicity $m_i$ then the corresponding general orbit is the sum $\sum m_i O(x_i)$ in the monoid generated by prime orbits. 
 
Let $\mathcal P$ denote the set of prime orbits, and $\mathcal N$ the set of general orbits, and for $X \geq 1$, let $\mathcal P(X), \mathcal N(X)$ denote the subset of all elements of length less than or equal to $X$ in the respectively sets. Let $N_n$ denote the total number of general orbits of length $n$, and $P_\ell$ the total number of prime orbits of length $\ell$. 
Define the corresponding cumulative counting functions as $$N_\sigma(X) \defeq \# \mathcal N(X), \ P_\sigma(X) \defeq \# \mathcal P(X) \mbox{ and }M_\sigma(X) \defeq \sum_{P \in \mathcal P(X)} \Lambda^{-\ell(P)} = \sum_{\ell \leq X} P_\ell/\Lambda^\ell.$$ 
This last function $M_\sigma(X)$ is the analogue of Mertens's counting function $\sum_{p \leq X} 1/p$ ($p$ prime), where we count every prime orbit $O$ with `norm' $\Lambda^{\ell(O)}$. 

\subsection*{Asymptotic counting} 
The crucial quantity is a certain Ces\`aro mean. Recall that the Ces\`aro mean $\mathcal C(a_k)$ of a sequence $\{a_k\}_{k \geq 1}$ of real numbers is defined as $$\mathcal C(a_k) \defeq \lim_{X \rightarrow + \infty} \frac{1}{X} \sum_{k \leq X} a_k.$$

\begin{introtheorem} \label{introthma} 
Suppose $\mathcal{X}$ is a set and $\sigma \colon \mathcal{X} \rightarrow \mathcal{X}$ a confined self-map. Assume that $\Lambda$ is finite and satisfies $\Lambda>1$, and assume that the Ces\`aro mean 
$B \defeq \mathcal C (\sigma_k/\Lambda^k)$ 
 exists and is strictly positive. Then the following asymptotics hold: 
\begin{equation} \tag{N} \label{N} N_\sigma(X) = \frac{C}{\Gamma(B)} \cdot \Lambda^X X^{B-1}\Big(1 + O\Big(\frac{1}{\log X}\Big)\Big) \end{equation} for some $C>0$, and 
\begin{equation} \tag{M} \label{M} M_\sigma(X) \sim B \log X+ o(\log X). \end{equation} 
\end{introtheorem}

To write the leading constant in \eqref{N} as $C/\Gamma(B)$ is natural from the proof and useful in the formulation of further results below. The proof of Theorem \ref{introthma} hinges on the relations between $\sigma_k, N_n$ and $P_\ell$ expressed via the Artin--Mazur dynamical zeta function $\zeta_\sigma(z) = \exp \sum \sigma_k z^k/k$ of $\sigma$, which is at the same time the exponential generating function for $\sigma_k$, the Taylor series for $N_n$ and the Euler product for $P_\ell$. The generating series has radius of convergence $1/\Lambda$ and in general has a natural boundary along $|z|=1/\Lambda$, but we can apply a logarithmic version of the Hardy--Littlewood real Tauberian theorem, taking the ---weak, but generally optimal--- error term from work of Freud \cite{Freud}. The rest of the proof involves standard applications of partial summation. The total count in \eqref{N} can have $B-1$ positive or negative (as seen in the Examples listed below), meaning that there may be polynomial excess or deficiency in the total count compared to dominant growth $\Lambda^X$.  For the quite different case where $\Lambda=1$, see Propositions \ref{Lambda1} and \ref{Lambda1a}. The case $\Lambda<1$ is not interesting, see Remark \ref{Lambdaless}. 

\subsection*{An LDP} As a consequence of Theorem \ref{introthma}, we prove a large deviation principle for the number of prime orbits occurring in the decomposition of a general orbit. 
\begin{introtheorem} \label{introthm:LDP}
With the same notations and conditions from Theorem \ref{introthma}, let $W$ be the random variable over $\Z_{>0}$ defined as the number of distinct prime orbits of a uniformly chosen element from the $\mathcal N(X)$ for each $X \in \Z_{>0}$. Then for any Borel measurable set $A \subset \R$ the probability $\mathbb{P} \left[ \frac{W(X)}{B \log X} \in A \right]$ satisfies a large deviation principle with speed $B \log X$ and rate function
	\begin{equation} \label{Ix}
		I(x) = \begin{cases}
			x \log x - x + 1 &\text{ if } x \geq 0 \\
			+\infty &\text{ otherwise}.
		\end{cases}
	\end{equation}
	In particular, for any Borel measurable set $A \subset \R$, we have
	\begin{equation}
	- \mathrm{inf}_{x \in A^\circ} I(x) \leq \liminf_{X \to \infty} \frac{\log \mathbb{P} \left[ \frac{W(X)}{B \log X} \in A \right]}{B \log X} \leq \limsup_{X \to \infty} \frac{\log \mathbb{P}\left[ \frac{W(X)}{B \log X} \in A \right]}{B \log X} \leq -\mathrm{inf}_{x \in \overline{A}} I(x).
	\end{equation}
\end{introtheorem}

The expected value of $W(X)$ equals $B \log X+o(\log X)$, which is a very weak analogue of the Hardy--Ramanujan theorem. About the higher moments, we can say nothing. Note that the speed is dependent on the Ces\`aro mean but the rate function $I$ in \eqref{Ix} is universal; it is the rate function of the standard Poisson distribution with unit mean; we do not know whether there is a deeper reason for why exactly the Poisson distribution shows up here. We prove a more technical Large Deviation Theorem for suitable strongly additive functions $g$ on the monoid of general orbits, see Theorem \ref{thm:LDP}. In this case, we need an assumption of convergence of the Laplace transform of counting measures related to $g$, and the resulting rate function depends on  $g$. The structure of the proof is very close to that of a similar result in \cite{KP}, which in turn follows the template from \cite{MehrdadZhu}, and is ultimately based on using the G\"artner--Ellis theorem.  A crucial difference between existing proofs and our arguments is found in the proof of Lemma \ref{LDP:lemma4}. As an application, we establish large deviation principles for a family of strongly additive functions whose value is zero over a positive density subset of $\mathcal{P}$, see Example \ref{example:LDP3}. We also construct functions for which the assumption of convergence of their related Laplace transform does not hold (see Example \ref{example:LDP1}), as well as functions whose associated rate function does not give a meaningful LDP (see Examples \ref{example:LDP3} and \ref{example:LDP4}).

\begin{introexample*} Since closed intervals are $I$-continuity sets, a more intuitive consequence of Theorem \ref{introthm:LDP} is the following statement about the rare event that an orbit has a fraction $\epsilon$ more than the expected number of prime orbit components: \[ \mbox{for any $\varepsilon>0$},  \mathbb{P} \left[ W(X) \geq (1+\varepsilon) B \log X \right] \sim 1/X^{B ((1+\varepsilon)\log(1+\varepsilon) -\varepsilon)+o(1)} \mbox{ as }X \rightarrow \infty.\] 
For small $\epsilon$ and large $X$, this is roughly $X^{-B\epsilon^2/2}$. 
\end{introexample*}

\subsection*{Applicability: algebraic group endomorphisms and beyond} We consider the applicability of Theorem \ref{introthma} and \ref{introthm:LDP} to concrete dynamical systems. 

\begin{introtheorem} \label{FADisgood} 
Suppose $(\mathcal{X},\sigma)$ is a FAD-system \textup{(}in the sense of Definition \ref{deffad} below\textup{)} with $\Lambda>1$. Then $\Lambda$ is finite and the Ces\`aro mean $B$ as in Theorem \ref{introthma} exists and satisfies $B>0$.  
\end{introtheorem}

To show the existence of the Ces\`aro mean for such systems, we use a Fourier analytic method. 
The somewhat exotic looking concept of FAD-system in fact embraces many natural examples: endomorphisms $\sigma$ of algebraic groups $G$ (abelian varieties, linear algebraic groups, \dots) over $\overline \F_p$, with $\mathcal{X}=G(\overline \F_p)$ \cite{BCH}; additive cellular automata $\sigma$ with $\mathcal{X}=(\F_p^r)^{\Z}$ (these turn out to be equivalent to vector group endomorphisms, see \cite{BCca}); as well as solenoids given as $S$-integer dynamical systems for global fields in the sense of Chothi--Everest--Ward \cite{Chothi} (see also \cite{LindWard}). Below are some concrete examples, each illustrating a different aspect. 

\begin{introexamples*} \mbox{ } 
\begin{enumerate}
\item[(FF)] $G=\Ga$ is the additive group over $\overline \F_q$ and $\sigma (x)=x^q$ is the $q$-Frobenius; then $\sigma_k = q^k$. We have $B=C=1$, and a primitive orbit of $\sigma^k$ corresponds to (a Galois orbit of a root of) an irreducible polynomial of degree $k$, so this case coincides with the well-known decomposition of a polynomial over a finite field $\F_q$ into irreducibles, already studied by Gau{\ss} in an unpublished chapter of the Disquisitiones, see \cite{Frei}.  
\item[(E)] $G=E$ is an ordinary elliptic curve over $\F_3$, and $\sigma(P)=2P$; then $\sigma_k = (2^k-1)^2\cdot |2^k-1|_3$. We have $B=5/2^3$ a rational number and $C/\Gamma(B) \approx 0.502128$ with $$C=2^{19/8}/3^{125/64} \prod_{j \geq 0} \left( \frac{4^{3^j}+1}{4^{3^j}-1}\right)^{1/3^{2j+1}}.$$
\item[(GA)] $G=\Ga$ is the additive group over $\F_2$ and $\sigma (x)= x^2+x$; then $\sigma_k = 2^{k-|k|_2^{-1}}$. We have $B = \sum_{j \geq 0} 2^{-1-j-2^j} \approx 0.320556$ (a transcendental number) and, with  $A \defeq \sum_{j \geq 0} j 2^{-1-j-2^j} \approx 0.078859$, then $C=2^{B-A+1}$ and $C/\Gamma(B) \approx 0.847382$. 
\item[(GM)] $G=\Gm^4$ is the multiplicative group over $\F_5$ and $\sigma(x_1,x_2,x_3,x_4)=(x_4^{-1}, x_1 x_4^3, x_2 x_4^{-3}, x_3 x_4^3)$; then $\sigma_k = |\det(A^k-1)| \cdot |\det(A^k-1)|_5$, where $A$ is the companion matrix of $x^4-3x^3+3x^2-3x+1$, the minimal polynomial of a Salem number. In this case, $B=2  \cdot 23^2 / (11 \cdot 71) \approx  1.354673.$
\end{enumerate}
Examples (FF) and (GA) are equivalent (in the sense of \cite{BCca}) to a cellular automaton, namely, (FF) is equivalent to the endomorphism of the shift space $\mathcal{X}=\F_q^{\Z}$ given by the shift itself $\sigma((y_i)_{i \in \Z}) = (y_{i+1})_{i \in \Z}$, and (GA) is equivalent to the endomorphism of the shift space $\mathcal{X}=\F_2^{\Z}$ given by $\sigma((y_i)_{i \in \Z}) = (y_i+y_{i+1})_{i \in \Z}$. 
\end{introexamples*}

For an algebraic group, the nature of the Ces\`aro limit $B$ depends on the underlying geometry. We will in fact compute $B$ for all integer multiplications on ordinary elliptic curves in characteristic $p\geq 3$ (Proposition \ref{BforE}), and for all additive cellular automata (Formula \eqref{BforCA}). We also prove a general result, using the following concept. In \cite[Prop.~8.1.2]{BCH} it is shown that for any algebraic group $G$ over $\overline \F_p$, there exists a descending series of fully characteristic subgroups (meaning subgroups $N$ for which $\sigma(N) \subseteq N$ for all $\sigma \in \End(G)$) in which all factor groups are either finite, a torus, a semisimple group, an abelian variety or a vector group. Call any such series an \emph{endoseries for $G$}. 

\begin{introtheorem} \label{dich} Suppose that $\sigma$ is a confined endomorphism of an algebraic group $G$ over $\overline \F_p$. Then 
\begin{enumerate}
\item the associated Ces\`aro mean $B$ is algebraic (and then, belongs to $\Q^\mathrm{ab}$) if the zeta function of the induced action on all vector group quotients occurring in an endoseries for $G$ is algebraic; 
\item the Ces\`aro mean $B$ is rational if, in addition, the \emph{fluctuation angles} for $\sigma$, in the sense of Definition \ref{fluctang}, are linearly independent over $\Q$;
\item the Ces\`aro mean $B$ for an endomorphism of a vector group with non-rational zeta function is transcendental. 
\end{enumerate} 
\end{introtheorem}

The algebraic and rational case are shown by direct calculation, and the complementary transcendence statement is shown using the method of Mahler functions. The theorem does not deal with groups whose endoseries has several occurrences of vector groups with non-rational zeta functions. 

\section{Asymptotic results for counting functions} \label{s1}

The goal of this section is to prove Theorem \ref{introthma}, namely, that existence and positivity of the Ces\`aro mean imply {\normalfont \eqref{N}} and {\normalfont \eqref{M}}. 

We fix $(\mathcal{X},\sigma)$ and leave out the subscript $\sigma$ from the quantities involved. We also set $N_0=1$. 
Consider the Artin--Mazur dynamical zeta function of the system $(\mathcal{X},\sigma)$ given by 
\begin{equation} \label{AM} \zeta(z) \defeq \exp ( \sum_{k \geq 1} \sigma_k z^k/k) = \sum_{n \geq 0} N_n z^n = \prod_{\ell \geq 1} (1-z^\ell)^{-P_\ell},
\end{equation}
see \cite{AM} or \cite{Smale}. By the Cauchy--Hadamard criterion and the definition of $\Lambda$, the function $\zeta(z)$ is holomorphic in the disk $|z|<1/\Lambda$, but in general may have a natural boundary on $|z|=1/\Lambda$, i.e, cannot analytically be extended to any open set intersecting the boundary; see, e.g., \cite[\S 11.3]{BCH} for examples. 
The last equality in \eqref{AM} is a formal consequence of the factorisation of general orbits into prime orbits. The equality between first and last term follows since every one of the $\ell$ points in a prime orbit of length $\ell$ is a fixed point of $\sigma^k$ precisely if $\ell$ divides $k$, i.e., we have the first equation below (and the second one follows by M\"obius inversion of the first):
 \begin{equation} \label{swapsp} \sigma_k = \sum_{\ell \mid k} \ell P_\ell \mbox{ and } P_\ell = \frac{1}{\ell} \sum_{n \mid \ell} \mu(l/n) \sigma_n. \end{equation}

\begin{proof}[Proof of Theorem \ref{introthma}] The first step in proving \eqref{N} is the observation that the existence of the Ces\`aro limit $B$ implies \emph{logarithmic summability}, in the following more general sense.
\begin{lemma} \label{logas} If $\{b_k\}_{k \geq 1}$ is a sequence of real numbers whose Ces\`aro mean $B=\mathcal C(b_k)$ exists, and $f(u)=\sum b_k u^k/k$ converges for $0 \leq u<1$, then $\displaystyle{\lim_{u \rightarrow 1^-} \frac{-1}{\log(1-u)} f(u) = B.} $ \end{lemma} 
\begin{proof}[Proof of Lemma \ref{logas}] 
Set $S(X) \defeq \sum_{k \leq X} b_k$. The hypothesis says that  \begin{equation} \label{defsx} S(X) = B X + \epsilon_X X \mbox{ with } \epsilon_X \rightarrow 0. \end{equation} 
 First, notice that by replacing $b_k$ with $b_k-B$, it suffices to prove the implication for $B=0$ (doing this changes $f$ to $f-\sum Bu^k/k = f-B\log(1-u)$). Since $S(X) u^X/X \rightarrow 0$, Abel summation applied to $f$ gives 
 $$ f(u) = \sum_{k \geq 1} S(k) \left( \frac{u^k}{k}-\frac{u^{k+1}}{k+1} \right) = \sum_{k \geq 1} S(k)u^k \underbrace{\frac{k+1-ku}{k(k+1)}}_{\geq 0}$$
 Given $\varepsilon>0$, choose $K$ sufficiently large with $|S(k)| \leq \epsilon k$ for all $k>K$ (which is possible since we now assume that $B=0$). Then 
$$ |f(u)| \leq \text{cst} + \epsilon \sum_{k \geq K} u^k \frac{k+1-ku}{(k+1)} \leq \text{cst} + \epsilon \sum_{k \geq 1} u^k \frac{k+1-ku}{(k+1)} = \text{cst} - \epsilon \log(1-u); $$
and thus \[\limsup_{u \rightarrow 1^-} \frac{-|f(u)|}{\log(1-u)} \leq \epsilon\] and the result follows by letting $\epsilon \rightarrow 0$.
\end{proof} 
 Since we assume $\Lambda$ to be finite and $B$ to exist, we can apply the lemma with $b_k = \sigma_k/\Lambda^k$, with the function $f(u) \defeq \log \zeta(u/\Lambda) = \sum \sigma_k/\Lambda^k u^k/k$ converges in the open disk $|u|<1$. 
By exponentiation, the logarithmic summability implies that, for $u \rightarrow 1^-$, 
\begin{equation} \label{c1} \zeta(u/\Lambda) = \sum N_n / \Lambda^n u^n \sim c' \cdot (1-u)^{-B} + O(1), \end{equation} for some $c'>0$. Now by the positivity $B>0$, the Hardy--Littlewood `real' Tauberian theorem with remainder term, due to G.~Freud \cite{Freud} (compare \cite{HL}, \cite[I.7.4, VI.\S 1]{Korevaar}), implies that 
\begin{equation} \label{preabel}  A(X) \defeq \sum_{n\leq X} N_n / \Lambda^n \sim c'' \cdot X^B + O\Big(\frac{X^B}{\log X} \Big),  \end{equation} 
for the constant $c'' = c'/\Gamma(B+1)$. We can convert the asymptotics in \eqref{preabel} into one for $N(X)$ by applying again Abel summation, which gives for the main term
$$ N(X) = A(X) \Lambda^X - \sum_{n=1}^{X-1} A(n) (\Lambda^{n+1}-\Lambda^n) \sim c'' X^B \Lambda^X - c''(\Lambda-1) \sum_{n=1}^{X-1} n^B \Lambda^n, $$
and with 
\begin{align*} \sum_{n=1}^{X-1} n^B \Lambda^n &\sim  \Lambda^X X^B \Bigg( \sum_{k=1}^{\lfloor \sqrt{X} \rfloor} \Lambda^{-k} \underbrace{\Big( 1 - \frac{k}{X} \Big)^B}_{=1-\frac{Bk}{X}+O(\frac{1}{X^2})} + \underbrace{\sum_{k=\lfloor \sqrt{X} \rfloor + 1}^{X-1}\Lambda^{-k} \Big( 1 - \frac{k}{X} \Big)^B \Bigg)}_{=O(\Lambda^{-\sqrt{X}})}  \\ 
& \sim \frac{\Lambda^X X^B}{\Lambda-1} \left( 1-\frac{B\Lambda}{X(\Lambda-1)} + O\Big(\frac{1}{X^2}\Big) \right)
\end{align*} 
we find 
 $N(X) \sim \sum_{n\leq X} N_n \sim C/\Gamma(B) \Lambda^X \cdot X^{B-1}$
where now 
\begin{equation} \label{c2} C=c' \Lambda/(\Lambda-1). \end{equation}
The same method works for the error term, noticing that in the corresponding terms of the form $$\sum_{k=1}^{X-1}\frac{1}{\log(X-k)} \cdot \Lambda^{-k}\Big( 1 - \frac{k}{X} \Big)^B,$$ we can replace ${1}/{\log(X-k)}$ by $1/\log(X)$ as $X \rightarrow + \infty$. 
Combining everything, we find \eqref{N}. 

\medskip

To prove \eqref{M}, we insert \eqref{swapsp} into the definition of $M(X)$ and split off a main term, to get
$$ M(X) = \sum_{\ell \leq X} \frac{P_\ell}{\Lambda^\ell} = \sum_{j \leq X} \frac{\sigma_j}{j \Lambda^j} + \underbrace{\sum_{d \geq 2} \frac{\mu(d)}{d} \sum_{j \leq X/d}  \frac{\sigma_j}{\Lambda^j} \frac{1}{j \Lambda^{j(d-1)}}}_{=R(X)}. $$ 
The second  term $R(X)$ is $O(1)$, as can be seen by using \eqref{defsx} and Abel summation giving $|R(X)|  = O(\sum_{d \geq 2} 1/(d\Lambda^{d-1})).$ On the other hand, for the first term, Abel summation gives
\begin{equation} \label{abj} \sum_{j \leq X} \frac{\sigma_j}{j \Lambda^j} = S(X)/X + \sum_{k=1}^{X-1} \frac{S(k)}{k(k+1)}. \end{equation} 
Now \eqref{defsx}  allows us to rewrite \eqref{abj} as $$ B+ \epsilon_X + B \sum_{k=1}^{X-1} \frac{1}{k+1} + \sum_{k=1}^{X-1} \frac{\epsilon_k}{k+1}  = B \log X + o(\log X), $$ implying the result. 
\end{proof} 

We add the corresponding quite different result for $\Lambda=1$ (i.e., the case of subexponential growth rate of the number of fixed points).

\begin{proposition} \label{Lambda1}
Suppose $\mathcal{X}$ is a set and $\sigma \colon \mathcal{X} \rightarrow \mathcal{X}$ a confined self-map for which $\limsup \log(\sigma_k)/k=0$, i.e., the growth rate $\Lambda$ equals $1$, and assume that the Ces\`aro mean $B \defeq \mathcal C (\sigma_k)$ 
 exists and is strictly positive. Then 
\begin{equation} \label{N1} \tag{N1} N_\sigma(X) =  \frac{C}{\Gamma(B+1)} \cdot X^B\Big(1 + O\Big(\frac{1}{\log X}\Big)\Big), \end{equation} 
for some $C>0$. Furthermore, the dynamical zeta function $\zeta_\sigma(z)$ either has a natural boundary along the unit circle $|z|=1$, or it is a rational function. In the latter case, $\sigma_k$ is periodic with period $\varpi$, $B$ is the actual average $B=\varpi^{-1} \sum_{k=0}^{\varpi-1} \sigma_k \in \Z_{>0}$, and $P_\ell=0$ unless $\ell \mid \varpi$. In particular, the set of prime orbits is finite of cardinality $B$, and
\begin{equation} \label{M1} \tag{M1} M_\sigma(X) = B \end{equation} 
for $X>\varpi$. Furthermore, in this case $C/\Gamma(B+1)=1/B! \cdot \prod_{\ell \mid \varpi} \ell^{-P_\ell}$  in \eqref{N1}. 
\end{proposition}

For the zeta function to have a \emph{natural boundary along $|z|=1$} means that it cannot be analytically continued to any open set intersecting the unit circle (see, e.g.\ \cite[Ch.~6]{Segal}). 

\begin{proof} 
For $\Lambda=1$, the arguments in the proof of Theorem \ref{introthma} up to Formula \eqref{preabel} remain valid, and imply \eqref{N1} with $C = c'/\Gamma(B+1)$, where $c'$ is determined by $\zeta(u) \sim c' \cdot (1-u)^{-B}$ as $u \rightarrow 1^-$. 

The dynamical zeta function $\zeta(z)$ has a Taylor series expansion with integer coefficients converging inside the unit disk, so by the P\'olya--Carlson theorem \cite[Ch.~6]{Segal}, either $\zeta(z)$ has a natural boundary on the unit circle, or it is a rational function. In the first case, there is little we can say, so assume we are in the second case. Then also $Z(z)=z \zeta'(z)/\zeta(z) = \sum_{k \geq 1} \sigma_k z^k$ is a rational function. Since $\sigma_k = \sum_{\ell \mid k} \ell P_\ell$, the generating series $Z(z)$ is also a Lambert series
$$ Z(z) \defeq \sum_{k \geq 1} \sigma_k z^k = \sum_{\ell \geq 1} \ell P_\ell \frac{z^\ell}{1-z^\ell}. $$
At any primitive $N$-th root of unity $\zeta_N$, the Lambert series has residue $-\zeta_N \sum_{\ell \geq 1} P_{\ell N}$. Since all $P_\ell$ are non-negative, as soon as $P_N \neq 0$, $Z(z)$ has a pole at $\zeta_N$. It follows that only finitely many $P_\ell$ are non-zero, and all poles of $Z(z)$ are at roots of unity; thus, we can write  $Z(z)=F(z)/(1-z^\varpi)$ for some integer $\varpi$ and a polynomial $F$. By expanding, it follows that $\{\sigma_k\}$ is periodic of period $\varpi$, and thus that $B$ is the actual average of $\{\sigma_k\}_{k=0}^{\varpi-1}$. The previous consideration of the Lambert series then also implies that $P_\ell \neq 0$ only if $\ell \mid \varpi$. 

We see that for $X\geq\varpi$, $M(X)$ is constant, equal to 
$ \sum_{\ell \mid \varpi} P_\ell.$ We claim that this is also equal to $B$, the average of $\sigma_k$, and we prove this by the simple computation 
\[ B = \frac{1}{\varpi} \sum_{k=1}^{\varpi} \sigma_k =  \frac{1}{\varpi} \sum_{k=1}^{\varpi} \sum_{\substack{\ell \mid k \\ \ell \mid \varpi}} \ell P_\ell
= \sum_{\ell \mid \varpi} P_\ell \frac{l}{\varpi} \sum_{\substack{1 \leq k \leq \varpi \\ \ell \mid k}} 1  = \sum_{\ell \mid \varpi} P_\ell \in \Z_{>0}.
\]
Now that we have established that the set of non-zero $P_\ell$ is finite, we can be more precise about \eqref{N1} in the case of a rational zeta function, namely,  
$$\zeta(z) = \prod\limits_{\ell \mid \varpi} (1-z^\ell)^{-P_\ell} \underset{z \rightarrow 1^-}{\sim} (1-z)^{-B} \prod\limits_{\ell \mid \varpi} \ell^{-P_\ell},$$
which gives the indicated value of $C/\Gamma(B+1)$.
\end{proof} 

\begin{remark} \label{Lambdaless} 
If $\Lambda<1$, we claim that $\sigma_k=0$ and $P_\ell=0$ for all $k$ and $\ell$. Indeed, $\sigma_k=0$ for all $k>N$ for some $N$ since $\limsup \log(\sigma_k)/k=0$ and $\sigma_k$ is a non-negative integer. But, as fix point sequence of a dynamical systems, $\sigma_k$ is also a Dold sequence (see, e.g., \cite{Dold}), meaning that the corresponding $P_\ell$ are integers for all $\ell$.  However, choosing $\ell=mp$ with fixed $m$ for which $P_m \neq 0$ and variable prime $p>N$, we find that $mp$ should divide $\sum_{k \mid \ell} \mu(\ell/k) \sigma_k = - P_m + \sum_{k \mid m} \mu(m/k) \sigma_{pk} = -P_m$, which is impossible for large enough $p$. 
\end{remark}

\section{Large deviation principles}

In this section, we show how Theorem \ref{introthma} implies Theorem \ref{introthm:LDP}. More generally, we will show large deviation principles for a wider class of strongly additive functions. We consider the following condition analogous to \cite[Cond.~1.3]{KP}.
To simplify notations in this section, we fix $\sigma$ and leave it out of the notation.

\begin{condition} \label{condition:LDP}
Let $g$ be a \emph{strongly additive} function $g: \mathcal N \to \R$, i.e. $g$ satisfies
\begin{align*}
g(P^k) &= g(P) \text{ for all prime orbits } P, \\
g(O_1 O_2) &= g(O_1) + g(O_2) \text{ for all coprime orbits } O_1, O_2,
\end{align*}
where two orbits $O_1$ and $O_2$ are \emph{coprime} if they have no common prime orbits occurring in their decompositions. 

Next, given any Borel measurable set $A \subset \R$, let $\rho_X$ be the probability measure given by
\begin{equation} \label{convmeas}
\rho_X(A)
= \frac{1}{M(X)} \sum\limits_{\substack{P \in \mathcal{P}(X)\\ g(P) \in A}} \hspace*{-0mm} {\Lambda^{-\ell(P)}}; 
\end{equation}
then we further suppose that there is a convergence of the Laplace transforms of the measures $\rho_X$ as $X \rightarrow +\infty$ of the form 
$$\int_{\R} e^{\theta y} \rho_X(dy) \to \int_{\R} e^{\theta y} \rho(dy)<\infty$$ for any $\theta \in \R$ and some probability measure $\rho$ on $\R$. 
\end{condition}

\begin{example} \label{example:LDP1} For a general orbit $O$ and a prime orbit $P$, denote by $P \mid O$ the property that $P$ occurs in the decomposition of $O$ into prime orbits.  Given a weight function $w \colon \mathcal P \rightarrow \Z_{>0}$, the weighted decomposition statistics $g(O) = \sum_{P \mid O} w(P)$ satisfies Conditions \ref{condition:LDP} if and only if the stated convergence holds for the measure in \eqref{convmeas} with the summation condition $g(P) \in A$ replaced by $w(P) \in A$. Choosing $w(P)=1$ gives the function used in Theorem \ref{introthm:LDP}. 
\end{example}

\begin{example}
We give an example of weights where Conditions \ref{condition:LDP} may not hold. For example, when $w(P) = \ell(P)$, then the Laplace transform of the measures $\rho_X$ is given by
\[ \int_{\R} e^{\theta y} \rho_X(dy) = \frac{1}{M(X)} \sum_{\substack{P \in \mathcal{P}(X)}} \frac{e^{\theta \ell(P)}}{\Lambda^{\ell(P)}}, \]
which, for instance, does not have a limit if $\theta > \log \Lambda$ as $X \to + \infty$. Other interesting examples where analogues of Conditions \ref{condition:LDP} may not hold for abelian monoids can be found in \cite[Rem.~4]{MehrdadZhu}.
\end{example}

\begin{theorem} \label{thm:LDP}
	Assume the notations and conditions from Theorem \ref{introthma}. Let $g$ be a strongly additive function satisfying Condition \ref{condition:LDP}.  Let $W$ be a random variable over $\Z_{>0}$ defined as $W(O) \defeq g(V(O))$, where $V(O)$ is a uniformly chosen element from the set $\mathcal N(X)$ for each $X \in \R_{\geq 0}$. Recall that $B$ is the Ces\`aro mean. Then for any Borel measurable set $A \subset \R$ the probability $\mathbb{P} \left[ \frac{W(X)}{B \log X} \in A \right]$ satisfies a large deviation principle with speed $B \log X$ and rate function
\begin{equation}
I(x) \defeq \sup_{\theta \in \R} \left\{ \theta x - \int_{\R}(e^{\theta y} - 1) \rho(dy) \right\}.
\end{equation}
In particular, for any Borel measurable set $A \subset \R$, we have
	\begin{equation}
	- \mathrm{inf}_{x \in A^\circ} I(x) \leq \liminf_{X \to \infty} \frac{\log \mathbb{P} \left[ \frac{W(X)}{B \log X} \in A \right]}{B \log X} \leq \limsup_{X \to \infty} \frac{\log \mathbb{P}\left[ \frac{W(X)}{B \log X} \in A \right]}{B \log X} \leq -\mathrm{inf}_{x \in \overline{A}} I(x).
	\end{equation}
	\end{theorem}

\begin{example} \label{example:LDP2}
In the case where $g$ counts the number of distinct prime orbits occurring in the decomposition of a general orbit $O$ (as in Theorem \ref{introthm:LDP}), the measure $\rho_X(A)$ is defined as $\rho_X(A) = 1$ if $1 \in A$, and $\rho_X(A) = 0$ if $1 \not\in A$. With this measure, the rate function is indeed as in \eqref{Ix}. 
\end{example} 

\begin{example} \label{example:LDP3}
One can use Theorem \ref{thm:LDP} to count the number of distinct prime orbits inside any subset $\mathcal{P}' \subset \mathcal{P}$ in the decomposition of a general orbit $O$. Denote by $\mathcal{P}'(X) := \mathcal{P}' \cap \mathcal{P}(X)$. Define $g$ to be a strongly additive function for some positive $\lambda > 0$:
\begin{equation}
g(P) = \begin{cases}
\lambda &\text{ if } P \in \mathcal{P'}, \\
0 &\text{ otherwise}.
\end{cases}
\end{equation}
Suppose that the following limit exists:
\[ \mathfrak{r} := \lim_{X \to + \infty} \frac{1}{M(X)} \sum_{P \in \mathcal{P}'(X)} \Lambda^{-\ell(P)}. \]
Then the limit of the Laplace transforms of the measures $\rho_X$ can be computed as
\[ \int_{\R} e^{\theta y} \rho_X(dy) = \frac{1}{M(X)} \sum_{P \in \mathcal{P}(X) \setminus \mathcal{P}'(X)} \Lambda^{-\ell(P)} + \frac{e^{\lambda \theta}}{M(X)} \sum_{P \in \mathcal{P}'(X)} \Lambda^{-\ell(P)} \underset{X \to +\infty}{\longrightarrow} (1-\mathfrak{r}) + e^{\lambda \theta} \mathfrak{r}. \]
Hence, one can take the measure $\rho$ characterised by $\rho(\{0\}) = 1 - \mathfrak{r}$ and $\rho(\{\lambda\}) = \mathfrak{r}$.
The rate function is
\begin{equation}
I(x) := \sup_{\theta \in \mathbb{R}} \Big\{ \theta x - e^{\lambda\theta} \mathfrak{r} + \mathfrak{r} \Big\} =
\begin{cases}
\frac{x}{\lambda} \log \Big(\frac{x}{\lambda\mathfrak{r}}\Big) - \frac{x}{\lambda} + \mathfrak{r} &\text{ if } \mathfrak{r} \neq 0 \text{ and } x \geq 0, \\
+ \infty &\text{ otherwise}.
\end{cases}
\end{equation}
In the case when $\mathfrak{r} \neq 0$, we have $I(\lambda \mathfrak{r}) = 0$. This reflects the fact that the analogue of {\normalfont \eqref{M}} of Theorem \ref{introthma} for the subset $\mathcal{P}'$ would be asymptotic to $B \lambda \mathfrak{r} \log X$ as $X$ grows arbitrarily large. On the other hand, when $\mathfrak{r} = 0$, then the rate function $I(x)$ is trivial, i.e. $I(x) = + \infty$.
\end{example}

\begin{example} \label{example:LDP4}
One can obtain trivial rate functions for strongly additive functions $g$ other than those constructed previously. For example, let $g(P) \coloneq \Lambda^{-\ell(P)}$ for all $P \in \mathcal{P}$. We use Taylor expansion to obtain
\[ \int_{\R} e^{\theta y} \rho_X(dy) = \frac{1}{M(X)} \sum_{\substack{P \in \mathcal{P}(X)}} \frac{e^{\Lambda^{-\ell(P)} \theta}}{\Lambda^{\ell(P)}} = \frac{1}{M(X)} \sum_{\substack{P \in \mathcal{P}(X)}} \sum_{k=0}^\infty \frac{1}{k!} \Lambda^{-(k+1) \ell(P)} \theta^k. \]
Using {\normalfont \eqref{M}} of Theorem \ref{introthma} implies that the above expression can be simplified to $1 + O(1/\log X)$. Taking $X \to +\infty$, we obtain that the rate function is given by $I(x) := \sup_{\theta \in \R} \{ \theta x \} = + \infty$.
\end{example} 

We prove the theorem by adapting the arguments presented in \cite{KP}; only Lemma \ref{LDP:lemma4} requires an entirely new argument. For each prime orbit $P$, we define the independent random variable $Y_P$ for which
\begin{equation}
	Y_P = 
	\begin{cases}
		1 &\text{ with probability } {\Lambda^{-\ell(P)}}, \\
		0 &\text{ otherwise }.
	\end{cases}
\end{equation}
For each $X \in \Z_{>0}$, we denote by $V(X)$ a uniformly chosen element from the set $\mathcal N(X).$ For each prime orbit $P$, we also define a random variable $Z_P$ such that
\begin{equation}
	Z_P =
	\begin{cases}
		1 &\text{ if } P \mid V(X), \\
		0 &\text{ otherwise }.
	\end{cases}
\end{equation}

\begin{lemma} \label{LPD:lemma1}
For distinct prime orbits $P_1, P_2, \cdots, P_k$, we have 
\[			\mathbb{E} \left[ \prod_{i=1}^k Z_{P_i} \right] = \frac{N(X - \sum_{i=1}^k \ell(P_i))}{N(X)} \quad \mbox{ and } \quad
			\mathbb{E} \left[ \prod_{i=1}^k Y_{P_i} \right] = {\Lambda^{-\sum\limits_{i=1}^k \ell(P_i)}}.
\]
\end{lemma}
\begin{proof}
	The lemma follows from using the definition of $Y_{P_i}$ and $Z_{P_i}$. For the first equality, note that for a general orbit $O = P_1 P_2 \cdots P_k O'$, $\ell(O) \leq X$ is equivalent to $\ell(O') \leq X-\sum \ell(P_i)$. 
\end{proof}

\begin{notation} $\displaystyle{k_X \defeq {X}/{\mathrm{exp}(\sqrt[4]{\log X})}}$. This is chosen so that $\log k_X = \log X + o(\log X)$, which is used in later proofs.
\end{notation} 

\begin{lemma} \label{LDP:lemma2}
For distinct prime orbits $P_1, P_2, \cdots, P_k$ and sufficiently large $X$, we have 
	\begin{enumerate}
		\item 
		$\displaystyle{
		\mathbb{E} \left[ \prod_{i=1}^k Z_{P_i} \right] = \Lambda^{-\sum\limits_{i=1}^k \ell(P_i)} \cdot \left(1 - \frac{\sum_{i=1}^k \ell(P_i)}{X} \right)^{B-1} \cdot \left( 1 + O\Big(\frac{1}{\log X}\Big)\right).
		}$
		\item 
		$\displaystyle{
			\mathbb{E} \left[ \prod_{i=1}^k Z_{P_i} \right] \leq \mathbb{E} \left[ \prod_{i=1}^k Y_{P_i} \right] + O\Big(\frac{1}{\log X}\Big).
		}$
		\item If all the prime orbits $P_1, \cdots P_k$ satisfy $\ell(P_i) \leq k_X$, then for any $k < K (\log X)^2$ for some fixed $K > 0$,
		\begin{equation*}
		\mathbb{E} \left[ \prod_{i=1}^k Z_{P_i} \right]  = \mathbb{E} \left[ \prod_{i=1}^k Y_{P_i} \right] \cdot \left( 1 + O\Big(\frac{1}{\log X}\Big)\right)
		\end{equation*}
		\item There exists an explicit constant $L > 0$ such that for any sequence of real non-negative numbers $\{\theta_P\}$ as $P$ varies over the set of prime orbits, we have
		\begin{equation*}
			\mathbb{E} \Big[ \mathrm{exp} \Big( \sum_{P \in \mathcal P(X)} \hspace*{-3mm} \theta_P Z_P \Big)\Big] \leq L \mathbb{E} \Big[ \mathrm{exp} \Big( \sum_{P \in \mathcal P(X)}   \hspace*{-3mm} \theta_P Y_P \Big)\Big].
		\end{equation*}
 	\end{enumerate}
\end{lemma}

\begin{proof} Since the statement involves only sufficiently large $X$, we may safely assume that $X>\sum_{i=1}^k \ell(P_i)$. 

Statement (i) follows from substituting \eqref{N} from Theorem \ref{introthma} into the numerator and denominator of the first formula of Lemma \ref{LPD:lemma1}.  

Statement (ii) follows from the fact that the exponent $B-1$ is independent of $X$. If $\sum_{i=1}^k \ell(P_i) = o ({X}/{\log X})$, then the statement follows from using the Taylor expansion for the term 
\[\Big(1 - \frac{\sum_{i=1}^k \ell(P_i)}{X} \Big)^{B-1} = \Big(1 - o\Big( \frac{1}{\log X} \Big) \Big)^{B-1} \] 
to prove the statement. Otherwise, the statement follows from the fact that  
\[\mathbb{E} \Big[ \prod_{i=1}^k Y_{P_i} \Big] = o \Big( \frac{1}{\log X} \Big). \]

To prove statement (iii), we observe that the given conditions on $\ell(P_i)$ and $k$ imply $\sum_{i=1}^k \ell(P_i)/X = o (1/{\log X})$. 
We use the Taylor expansion as in the proof of statement (ii) to show the desired statement. 

Statement (iv) follows by inserting the Taylor expansion of the exponential function and using the following identities that hold for any set of non-negative integers $\{r_i\}_{i=1}^k$:
\[ 
	\mathbb{E} \left[ \prod_{i=1}^k Z_{P_i}^{r_i} \right] = \mathbb{E} \left[ \prod_{i=1}^k Z_{P_i} \right], \hspace{15pt} \mathbb{E} \left[ \prod_{i=1}^k Y_{P_i}^{r_i} \right] = \mathbb{E} \left[ \prod_{i=1}^k Y_{P_i} \right]. \qedhere
\]
\end{proof}

We state the exponential Chebyshev/Markov Inequality, which we will use in upcoming lemmas.
\begin{lemma}[{\cite[{Ch.\ I, \S6, (34)}]{Shiryaev}}] \label{LDP:expCheb}
Let $X$ be an integrable random variable whose variance is non-zero and finite. Then for any $\epsilon > 0$ and any $\theta > 0$, 
\[
\pushQED{\qed} 
\log \mathbb{P}[X \geq \epsilon] \leq \log \mathbb{E}\big[ e^{\theta X} \big]-\theta \epsilon.
\qedhere
\popQED
\]  

\end{lemma}

\begin{notation} $\displaystyle{\tilde{k}_X \defeq X^{\frac{1}{\log \log X}}}.$ This is chosen, so that $\log \tilde{k}_X = o(\log(X))$, which is used in later proofs.
\end{notation} 

We first show that it suffices to consider the problem over subset of orbits $O$ such that all its prime components $P$ satisfy $|g(P)| \leq C$ for some sufficiently large $C$.
\begin{lemma} \label{LDP:lemma3-0}
For any $\epsilon > 0$, we have
\begin{equation}
\limsup_{C \to \infty} \limsup_{X \to \infty} \frac{1}{B \log X} \log \mathbb{P} \Bigg[ \Big| \sum_{\substack{P \in \mathcal P(X), |g(P)| > C}} \hspace*{-5mm} g(P) Z_P \Big| \geq \epsilon B \log X \Bigg] = - \infty.
\end{equation}
\end{lemma}
\begin{proof}
We will show that
\begin{equation} \label{limsup} 
\limsup_{C \to \infty} \limsup_{X \to \infty} \frac{1}{B \log X} \log \mathbb{P} \Bigg[ \sum_{\substack{P \in \mathcal P(X), |g(P)| > C}} \hspace*{-5mm}  g(P) Z_P \geq \epsilon B \log X \Bigg] = - \infty.
\end{equation}
One then shows in a similar way that 
\begin{equation*}
\limsup_{C \to \infty} \limsup_{X \to \infty} \frac{1}{B \log X} \log \mathbb{P} \Bigg[  \sum_{\substack{P \in \mathcal P(X), |g(P)| > C}} \hspace*{-5mm}  g(P) Z_P \leq -\epsilon B \log X \Bigg] = - \infty
\end{equation*}
so that the lemma follows. For the proof of \eqref{limsup}, we reason as follows: for any $\theta, \epsilon>0$, we can estimate
\begin{align*}
\begin{split}
&\limsup_{X \to \infty} \frac{1}{B \log X} \log \mathbb{P} \Bigg[ \sum_{\substack{P \in \mathcal P(X), |g(P)| > C}} \hspace*{-5mm}  g(P) Z_P \geq \epsilon B \log X \Bigg] \\
\overset{\ref{LDP:expCheb}}&{\leq} \limsup_{X \to \infty} \frac{1}{B \log X} \log \mathbb{E} \Bigg[ \mathrm{exp} \Big( \theta \hspace*{-5mm} \sum_{\substack{P \in \mathcal P(X), |g(P)| > C}} \hspace*{-5mm} g(P) Z_P \Big) \Bigg] - \theta \epsilon \\
\overset{ \ref{LDP:lemma2}\text{(iv)}}&{\leq}  \limsup_{X \to \infty} \frac{1}{B \log X} \log \mathbb{E} \Bigg[ \mathrm{exp} \Big( \theta \hspace*{-5mm}  \sum_{\substack{P \in \mathcal P(X), |g(P)| > C}}  \hspace*{-5mm} g(P) Y_P \Big) \Bigg] - \theta \epsilon.
\end{split}
\end{align*}
We use formula \eqref{M} from Theorem \ref{introthma} to replace the leading factor $1/B \log X$ in this formula by the counting function $M(X)$, and find that the above quantity can be further bounded by
\begin{align*}
&\leq \limsup_{X \to \infty} \frac{1}{M(X)}{\sum\limits_{\substack{P \in \mathcal P(X), |g(P)| > C}} \hspace*{-5mm}  \log ((e^{\theta g(P)} - 1) {\Lambda^{-\ell(P)}} + 1)}- \theta \epsilon \\
&\leq \limsup_{X \to \infty} \frac{1}{M(X)} {\sum\limits_{\substack{P \in \mathcal P(X), |g(P)| > C}}  (e^{\theta g(P)} - 1) {\Lambda^{-\ell(P)}}}- \theta \epsilon,
\end{align*}
replacing $\log(1+x)$ by $x$ for small $x$. 
Using Condition \ref{condition:LDP}, the above expression is further bounded by 
\begin{align*}
&\leq \limsup_{X \to \infty} \int_{|y| > C} (e^{\theta y} - 1) \rho_X(dy) - \theta \epsilon \leq  \int_{|y| > C} (e^{\theta y} - 1) \rho(dy) - \theta \epsilon.
\end{align*}
Because $\int_{|y| > C} (e^{\theta y} - 1) \rho(dy) \to 0$ as $C \to \infty$, it follows that 
\begin{equation*}
\limsup_{C \to \infty} \limsup_{X \to \infty} \frac{1}{B \log X} \log \mathbb{P} \Big[ \sum_{\substack{P \in \mathcal P(X), |g(P)| > C}} \hspace*{-5mm} g(P) Z_P \geq \epsilon B \log X \Big] = - \infty. \qedhere
\end{equation*}
\end{proof}
Given $C>0$, we define the following two sets: 
\begin{align}
	& \mathcal A(X,C) \defeq \left\{ P \in \mathcal P \colon k_X \leq \ell(P) \leq X \text{ or } \ell(P) < \tilde{k}_X, \text{ and } |g(P)| \leq C \right\}, \\
	&  \mathcal B(X,C) \defeq \left\{ P \in \mathcal P \colon \tilde{k}_X \leq \ell(P) \leq k_X, \text{ and } |g(P)| \leq C \right\}.
\end{align}
We first show that, for proving large deviation principles, contributions from prime orbits lying in the set $\mathcal A(X,C)$  can be ignored.
\begin{lemma} \label{LDP:lemma3}
	For any $\epsilon > 0$ we have
	\begin{equation}
	\limsup_{X \to \infty} \frac{1}{B \log X} \log \mathbb{P}\Bigg[ \Big| \sum_{P \in \mathcal A(X,C)} \hspace*{-3mm} g(P) Z_P \Big| \geq B \epsilon \log X \Bigg] = -\infty.
	\end{equation}
\end{lemma}
\begin{proof}
	Using statement (iv) of Lemma \ref{LDP:lemma2}, for any $\theta > 0$ and sufficiently large $X$ there exists an explicit constant $L> 0$ such that
	\begin{align*}
	& \log  \mathbb{E} \Bigg[ \mathrm{exp} \Big( \theta \cdot \big| \hspace*{-3mm} \sum_{P \in \mathcal A(X,C)} \hspace*{-3mm} g(P) Z_P \big| \Big) \Bigg] 
	\leq \log \mathbb{E} \Bigg[ \mathrm{exp} \Big( \theta \cdot \big| \hspace*{-3mm} \sum_{P \in \mathcal A(X,C)} \hspace*{-3mm} g(P) Y_P \big| \Big) \Bigg] + \log L \\
	&\leq \sum_{P \in \mathcal A(X,C)} \hspace*{-3mm} \log \Big( (e^{\theta g(P)}-1){\Lambda^{-\ell(P)}} + 1 \Big) + \log L \leq (e^{C\theta} - 1) \cdot \underbrace{\sum_{\substack{\ell < \tilde{k}_X \text{ or } \\ k_X \leq \ell \leq X}}\frac{P_\ell}{\Lambda^{\ell}}}_{=: \Sigma} + \log L = o(\log X). 
	\end{align*}
Indeed, by using Formula \eqref{M} from Theorem \ref{introthma}, we find 	
	\begin{align*}
		\Sigma &\leq M(X)-M(k_X) + M(\tilde{k}_X) \\
		 & = B \log X +o(\log X) - B \log k_X + \underbrace{o(\log k_X) + B \log \tilde{k}_X  + o(\log \tilde{k}_X)}_{=o(\log X)} \\
		 & = B \log X - B \log X + \sqrt[4]{\log X} + o(\log X) = o(\log X). 
	\end{align*}
Using Lemma \ref{LDP:expCheb}. we conclude
	\begin{align*}
	 \limsup_{X \to \infty} & \frac{1}{B \log X}\log \mathbb{P}\Bigg[ \Big| \sum_{P \in \mathcal A(X,C)} \hspace*{-3mm} g(P) Z_P \Big| \geq B \epsilon \log X \Bigg]  \\
	& \leq \limsup_{X \to \infty} \frac{1}{B \log X} \log  \mathbb{E} \Bigg[ \mathrm{exp} \Big( \theta \cdot \big| \hspace*{-3mm} \sum_{P \in \mathcal A(X,C)} \hspace*{-3mm} g(P) Z_P \big| \Big) \Bigg]  - \theta \epsilon 
	\leq -\theta \epsilon.
	\end{align*}
	We can then let $\theta$ to grow arbitrarily large to prove the lemma.
\end{proof}

Next, we show that the contributions from prime orbits lying in the sets $\mathcal B(X,C)$ can be understood by using the random variables $Y_P$. At this point, our arguments are completely different from the ones used in \cite[Lem.~3.7]{KP}, that are not applicable to our setup, because the error term in \eqref{N} is too weak. 

\begin{lemma} \label{LDP:lemma4}
	For any $\theta \in \R$, we have
	\begin{equation}
	\lim_{X \to \infty} \frac{1}{B \log X} \Bigg( \log \mathbb{E} \Big[ \mathrm{exp} \big( \theta \hspace*{-3mm} \sum_{P \in \mathcal B(X,C)} \hspace*{-3mm} g(P) Z_P \big) \Big] - \log \mathbb{E} \Big[ \mathrm{exp} \big( \theta \hspace*{-3mm} \sum_{P \in \mathcal B(X,C)} \hspace*{-3mm} g(P) Y_P \big) \Big] \Bigg) = 0.
	\end{equation}
\end{lemma}
\begin{proof}
	We will focus on understanding the expression 
	\begin{align} \label{LDP:eq1}
	\begin{split}
\mathbb{E} \Big[ \mathrm{exp} \big( \theta \hspace*{-3mm} \sum_{P \in \mathcal B(X,C)} \hspace*{-3mm} g(P) Z_P \big) \Big]. 
	\end{split}
	\end{align}
We first study, for any $K>1$,  the sum 
\begin{equation} \label{LDP:eq1-1} \sum_{r=0}^{\lfloor K(\log X)^2 \rfloor} \hspace*{-3mm} \sum_{\substack{\mathbf P \subset \mathcal B(X,C) \\ \# \mathbf P=r}} \hspace*{-1mm} \mathbb{P}\left[ \prod_{P \in \mathbf P} Z_{P} = 1 \right] \cdot \mathbb{P}\left[ \begin{array}{cc} Z_{P} = 0 \\ \forall P \in \mathcal B(X,C) \setminus \mathbf P \end{array} \right] \cdot e^{ \sum\limits_{P \in \mathbf P} g(P) \theta}.
\end{equation} 	
	Fix a finite set $\mathbf P=\{P_1,\dots,P_r\} \subset \mathcal B(X,C)$. Because $r < K (\log X)^2$, Equation \eqref{M} and the holomorphicity of the Artin-Mazur dynamical zeta function at $z = 1/\Lambda^2$ imply
	\begin{align*}
	&\sum_{\substack{P \in \mathcal B(X,C) \\ P \not\in \mathbf P}} \hspace*{-3mm} \log \left( 1 - \frac{N(X - \ell(P))}{N(X)} \right)^{-1} 
	= \sum_{\substack{P \in \mathcal B(X,C) \\ P \not\in \mathbf P}} \frac{N(X - \ell(P))}{N(X)} + O(1) \\
	&= (1 + O(\frac{1}{\log X})) \sum_{\substack{P \in \mathcal B(X,C) \\ P \not\in \mathbf P}} \frac{1}{\Lambda^{\ell(P)}} + O(1) 
	= (1 + O(\frac{1}{\log X})) \hspace*{-3mm} \sum_{\substack{P \in \mathcal B(X,C) \\ P \not\in \mathbf P}}\hspace*{-3mm} \log \left(1 - \frac{1}{\Lambda^{\ell(P)}}\right)^{-1}.
	\end{align*}
	Hence we have
	\begin{align*}
	 & \mathbb{P}\left[ \begin{array}{cc} Z_{P} = 0 \\ \forall P \in \mathcal B(X,C) \setminus \mathbf P \end{array} \right]  
	= \prod_{\substack{P \in \mathcal B(X,C) \\ P \not\in \mathbf P}}  \left(1 - \frac{N(X - \ell(P))}{N(X)} \right) \\
	&= (1 + O(1/\log X)) \cdot \prod_{\substack{P \in \mathcal B(X,C) \\ P \not\in \mathbf P}}  \left(1 - \frac{1}{\Lambda^{\ell(P)}} \right)
	= (1 + O(1/\log X)) \cdot \mathbb{P}\left[ \begin{array}{cc} Y_{P} = 0 \\ \forall P \in \mathcal B(X,C) \setminus \mathbf P \end{array} \right].
	\end{align*}
	Using Statement (iii) of Lemma \ref{LDP:lemma2}, we have
\[ (\ref{LDP:eq1-1}) = \hspace*{-5mm} \sum_{r=0}^{\lfloor K(\log X)^2 \rfloor} \hspace*{-3mm} \sum_{\substack{\mathbf P \subset \mathcal B(X,C) \\ \# \mathbf P=r}} \hspace*{-1mm} \mathbb{P}\left[ \prod_{P \in \mathbf P} Y_{P} = 1 \right] \cdot \mathbb{P}\left[ \begin{array}{cc} Y_{P} = 0 \\ \forall P \in \mathcal B(X,C) \setminus \mathbf P \end{array} \right]  \cdot e^{ \sum\limits_{P \in \mathbf P} g(P) \theta}\cdot \Big(1 + O\big(\frac{1}{\log X}\big)\Big).
\]
Thus, we have been able, in the regime of small $r$, to replace $Z_P$ by $Y_P$ with a controlled error term. 
	
	Next we consider the complementary sum
	\begin{equation} \label{LDP:eq1-2}
 \sum_{r>K(\log X)^2} \,  \sum_{\substack{\mathbf P \subset \mathcal B(X,C) \\ \# \mathbf P=r}} \hspace*{-1mm} \mathbb{P}\left[ \prod_{P \in \mathbf P} Z_{P} = 1 \right] \cdot \mathbb{P}\left[ \begin{array}{cc} Z_{P} = 0 \\ \forall P \in \mathcal B(X,C) \setminus \mathbf P \end{array} \right] \cdot e^{ \sum\limits_{P \in \mathbf P} g(P) \theta}.
	\end{equation}
Setting $C^* = -C$ if $\theta < 0$, and $C^* = C$ if $\theta \geq 0$, with $-C \leq g(P) \leq C$, we have 
	\begin{align*}
	(\ref{LDP:eq1-2}) &\leq \sum_{r>K(\log X)^2}\, \sum_{\substack{\mathbf P \subset \mathcal B(X,C) \\ \# \mathbf P=r}} \mathbb{P}\left[ \prod_{P \in \mathbf P} Z_{P} = 1  \right] e^{C^* r \theta}.
	\end{align*}
We now overcount using 
$$ \sum_{\substack{\mathbf P \subset \mathcal B(X,C) \\ \# \mathbf P=r}}   \mathbb{P}\left[ \prod_{P \in \mathbf P} Z_{P} = 1  \right] = \sum_{\substack{\mathbf P \subset \mathcal B(X,C) \\ \# \mathbf P=r}}  \hspace*{-3mm}\Lambda^{-\sum\limits_{P \in \mathbf P} \ell(P)} \leq \frac{1}{r!} \Big( \sum_{P \in \mathcal B(X,C)}  \Lambda^{-\ell(P)} \Big)^r. $$
We now use \eqref{M} again for this sum, and we use Stirling's approximation for $r!$ with $r>K (\log X)^2$, to find 
\begin{align*} 
	(\ref{LDP:eq1-2})  &\leq \sum_{r > K (\log X)^2} \frac{1}{r!} (B \log X + o(\log X))^r e^{C^* r \theta} \\
	&\leq \sum_{r > K (\log X)^2} \left( \frac{e^{C^* \theta + 1} B}{K \log X} \right)^r = O \left( \left(\frac{e^{C^* \theta + 1} B}{K \log X} \right)^{K (\log X)^2} \right).
	\end{align*}
	Therefore, we have for some $K > 1$,
	\begin{align*}
	(\ref{LDP:eq1}) &= (\ref{LDP:eq1-1}) + (\ref{LDP:eq1-2}) 
		= \mathbb{E} \Big[ \mathrm{exp} \big( \theta \hspace*{-3mm} \sum_{P \in \mathcal B(X,C)} \hspace*{-3mm} g(P) Y_P \big) \Big]\cdot \Big(1 + O\big(\frac{1}{\log X}\big)\Big) + O \left( \left(\frac{e^{C^* \theta + 1} B}{K \log X} \right)^{K (\log X)^2} \right).
	\end{align*}
	Hence, as long as we have for sufficiently large $X$,
	\begin{equation} \label{LDP:eq1-3}
	\left(\frac{e^{C^* \theta + 1} B}{K \log X} \right)^{K (\log X)^2} = o \Bigg(\mathbb{E} \Big[ \mathrm{exp} \big( \theta \hspace*{-3mm} \sum_{P \in \mathcal B(X,C)} \hspace*{-3mm} g(P) Z_P \big) \Big]\cdot (\log X)^{-1} \Bigg),
	\end{equation}
	we can guarantee that for sufficiently large $X$, 
	\begin{equation}
	(\ref{LDP:eq1}) = \mathbb{E} \Big[ \mathrm{exp} \big( \theta \hspace*{-3mm} \sum_{P \in \mathcal B(X,C)} \hspace*{-3mm} g(P) Y_P \big) \Big]\cdot (1 + O(1/\log X));
	\end{equation}
the statement of the lemma then follows by dividing both sides by $\mathbb{E} \Big[ \mathrm{exp} \big( \theta \hspace*{-3mm} \sum\limits_{P \in \mathcal B(X,C)} \hspace*{-3mm} g(P) Y_P \big) \Big]$ and then taking logarithms on both sides of the  resulting equation. Finally, to show equation (\ref{LDP:eq1-3}), consider the following lower bound, which follows from the definition and \eqref{M}.  
	\begin{align*}
	\mathbb{E} \Big[ \mathrm{exp} \big( \theta \hspace*{-3mm} \sum\limits_{P \in \mathcal B(X,C)} \hspace*{-3mm} g(P) Y_P \big) \Big] & \geq \hspace*{-3mm} \prod_{P \in \mathcal B(X,C)} \hspace*{-3mm} \mathbb{P}[Y_P = 0] = \hspace*{-3mm} \prod_{P \in \mathcal B(X,C)} \left(1 - \frac{1}{\Lambda^{\ell(P)}} \right) \\ & = O \left( \frac{1}{\mathrm{exp}(B \log X + o(\log X))} \right) = O(X^{-B + o(1)}).
	\end{align*}
	Because we have for sufficiently large $X$ and any small enough $\epsilon > 0$, 
	\begin{equation*}
	\left(\frac{e^{C^* \theta + 1} B}{K \log X} \right)^{K (\log X)^2} = O \left(X^{-K (1- \epsilon) \cdot \log X \cdot \log \log X} \right),
	\end{equation*}
	it follows that equation (\ref{LDP:eq1-3}) holds for sufficiently large $X$. 
\end{proof}

We state the G\"artner-Ellis theorem, which we will use the prove Theorem \ref{thm:LDP}.
\begin{theorem}[{\cite[p.~152]{MehrdadZhu}, \cite[Thm.~3.2]{KP}}] \label{thm:GE} 
	Given a sequence of random variables $Z_n$ over $\R$, and a sequence of positive numbers $a_n$ such that $\lim_{n \to \infty} a_n = \infty$, suppose that the following limit
	\begin{equation*}
	\Lambda(\theta) \defeq \lim_{n \to \infty} \frac{1}{a_n} \log \mathbb{E}[\exp (\theta a_n Z_n)]
	\end{equation*}
	exists and is differentiable for every $\theta \in \R$. Then for any Borel measurable set $A \subset \R$, the probability $\mathbb{P}[Z_n \in A]$ satisfies a large deviation principle with speed $a_n$ and rate function $I(x) \defeq \sup_{\theta \in \R} \{\theta x - \Lambda(\theta)\}$.
\end{theorem}

\begin{proof}[Proof of Theorem \ref{thm:LDP}]
For any real number $\theta$ we use Taylor expansion to show
\begin{align*}
\begin{split}
\frac{1}{B \log X} \log \mathbb{E} \Big[ \mathrm{exp} \big( \theta \hspace*{-3mm} \sum\limits_{P \in \mathcal B(X,C)} \hspace*{-3mm} g(P) Y_P \big) \Big] &= \frac{1}{B \log X}\log \hspace*{-3mm} \prod_{P \in \mathcal B(X,C)} \hspace*{-3mm} \mathbb{E} \left[ \mathrm{exp} (\theta g(P) Y_P) \right] \\
&= \frac{1}{B \log X} \hspace*{-1mm} \sum_{P \in \mathcal B(X,C)} \hspace*{-3mm} \log \Big( \frac{1}{\Lambda^{\ell(P)}} (e^{\theta g(P)} -1) + 1 \Big) \\
&= \frac{1}{B \log X} \Big(\sum_{P \in \mathcal B(X,C)} \frac{1}{\Lambda^{\ell(P)}} (e^{\theta g(P)} - 1) + O \big( \frac{P_\ell}{\Lambda^{2\ell}} \big) \Big).
\end{split}
\end{align*}
Statement \eqref{M} from Theorem \ref{introthma} can be used to replace $B \log X$ in the denominator by a prime count, so we find 
\begin{align*}
&  \lim_{X \to \infty} \frac{1}{B \log X} \log \mathbb{E} \Big[ \mathrm{exp} \big( \theta \hspace*{-3mm} \sum\limits_{P \in \mathcal B(X,C)} \hspace*{-3mm} g(P) Y_P \big) \Big]  \\
&= \lim_{X \to \infty} \frac{1}{M(X)} \sum\limits_{P \in \mathcal B(X,C)} {\Lambda^{-\ell(P)}} (e^{\theta g(P)} - 1) = \int_{-C}^C (e^{\theta y} - 1) \rho(dy),
\end{align*}
by Condition \ref{condition:LDP}. We note that one can always choose an increasing infinite sequence of values of $C$ to satisfy $\rho(\{C\}) = \rho(\{-C\}) = 0$. Then we use Lemma \ref{LDP:lemma4} to obtain
\begin{equation*}
\lim_{X \to \infty} \frac{1}{B \log X} \log\mathbb{E} \Big[ \mathrm{exp} \big( \theta \hspace*{-3mm} \sum\limits_{P \in \mathcal B(X,C)} \hspace*{-3mm} g(P) Z_P \big) \Big]= \int_{-C}^C (e^{\theta y} - 1) \rho(dy).
\end{equation*}
By the G\"artner-Ellis Theorem (see for example \cite[p.~152]{MehrdadZhu} or \cite[Thm.~3.2]{KP}), the probability $$\mathbb{P} \left[ \frac{1}{B \log X}\sum\limits_{P \in \mathcal B(X,C)} g(P) Z_P \in A \right]$$ satisfies a large deviation principle with rate function
\begin{equation}
I_C(X) \defeq \sup_{\theta \in \R} \left( \theta x - \int_{-C}^C (e^{\theta y} - 1) \rho(dy) \right).
\end{equation}
We now use Lemmas \ref{LDP:lemma3-0} and \ref{LDP:lemma3} and let $C \to \infty$ to obtain that the probability $\mathbb{P} \left[ \frac{W(X)}{B \log X} \in A\right]$ satisfies a large deviation principle with rate function $I(x) = \lim\limits_{C \to \infty} I_C(x)$. 
\end{proof}

\begin{remark} 
For $\Lambda=1$ with rational dynamical zeta function, the set of prime orbits is finite, of cardinality $B$. This implies that one expects $B$ prime orbits in the decomposition of any sufficiently large general orbit. Also, the probability that there are more than $B$ components is zero: $\mathbb{P}[W(X)>B]=0$ for all $X$. In this situation, there is no meaningful LDP. In case $\Lambda=1$ and the dynamical zeta function is irrational, we do not know whether interesting statistical results exist. 
\end{remark}

\begin{remark} 
The above arguments allow us to compute the expected value of $W$, as follows: \begin{align*} \mathbb{E}[W] &= \mathbb{E}[\hspace*{-2mm} \sum_{\ell(P) \leq X} \hspace*{-2mm}  Z_P] =  \mathbb{E}\Big[\hspace*{-2mm}  \sum_{\ell(P) \leq X} \hspace*{-2mm}  Y_P(1+O\Big(\frac{1}{\log X}\Big)\Big] = M(X)\Big(1+O\Big(\frac{1}{\log X}\Big)\Big) \\ &=  B \log X + o(\log X). \end{align*} 
This is a very weak analogue of the result of Hardy--Ramanujan. 
\end{remark}

\section{The Ces\`aro mean for FAD-systems} 

We recall the general context of FAD-systems from \cite{BCH}, and prove Theorem \ref{FADisgood}, i.e.,  that such systems satisfy the conditions of Theorem \ref{introthma}. 

\begin{definition} \label{deffad} A sequence $(a_n)$ of integers is called a \emph{gcd-sequence} if $\text{gcd}(a_m,a_n)=a_{\text{gcd}(m,n)}$ for all integers $m,n$. 
Let $\mathcal{X}$ denote a set, and $\sigma \colon \mathcal{X} \rightarrow \mathcal{X}$ a confined self map. We call $(\mathcal{X},\sigma)$ a \emph{FAD-system} if the numbers $\sigma_k$ form a FAD-sequence, meaning that
\begin{equation} \label{fadseq} \sigma_k = c^k |\det(A^k-1)| r_k \prod_{p \in S} |k|_p^{s_{p,k}} p^{-t_{p,k} |k|_p^{-1}} \end{equation}
where $c \in \R_{>0}$, $A$ is an integral matrix (with the convention that the entire factor $|\det(A^k-1)|$ may be absent), $r_k \in \R_{>0}$ is a periodic gcd-sequence, $S$ is a finite set of primes, and $s_{p,k}, t_{p,k} \in \R_{\geq 0}$ are periodic gcd-sequences of period coprime to $p$.  
\end{definition} 

\begin{remark} Not all sequences of integers of the form \eqref{fadseq} can be realised by a dynamical system. This `realisation problem' is discussed in \cite{BCH}. 
One can show, for example, that for realisable sequences, $c \in \Z_{>0}$ and \emph{if} $t_{p,k} \in \Z_{ \geq 0}$ for all $p,k$, then $r_k$ and $p^{s_{p,k}}$ take rational values \cite[\S 10.2]{BCH}. 
\end{remark} 

This somewhat exotic looking concept in fact embraces the following natural examples:  
\begin{enumerate}
\item Confined endomorphisms $\sigma$ of algebraic groups $G$ (abelian varieties, linear algebraic groups, \dots) over $\overline \F_p$, then $\mathcal X = G(\overline \F_p)$ and $S=\{p\}$; see \cite[Thm.~8.2.1]{BCH}.
\item Additive cellular automata $\sigma$ with $\mathcal X = V^{\Z}$ and  $V=\F_p^r$ (i.e., continuous---for the product topology of the discrete topology on $V$---such maps that are additive and commute with the operation of shifting the bi-infinite sequence one to the right); again with $S=\{p\}$, $c=p^a$ for some integer $a \geq 0$, the determinant factor is absent, $r_n=1$, $s_n=0$; see \cite[Prop.~9.2.3]{BCH}. In fact, this collection of examples corresponds in a precise sense to endomorphisms defined over $\F_p$ of vector groups $\Ga^r$, see \cite{BCca}.  
\item Solenoids given as $S$-integer dynamical systems for global fields in the sense of Chothi--Everest--Ward \cite{Chothi} with finite $S$, where $S$ is indeed $S$ for a number field, and $t_{p,n}=0$; and $S=\{p\}$ for a function field of characteristic $p>0$; see \cite[Prop.~9.1.4]{BCH}. Mertens-style theorems in this setup are discussed in \cite{Everest}. 
\end{enumerate}

\begin{proof}[Proof of Theorem \ref{FADisgood}] We use the following asymptotic result about $\sigma_k$, see \cite[Thm.~12.3.1(i)]{BCH} 
\begin{equation} \label{start0}    \sigma_k = \underbrace{4^m \prod_{j=1}^m \sin^2(k \theta_j/2)   r_k \prod_{p \in S} |k|_p^{s_{p,k}} p^{-t_{p,k} |k|_p^{-1}}}_{= b_k} \Lambda^k + O(\Lambda^{\vartheta k}) \mbox{ for some }\vartheta<1,   \end{equation}
where $\{ e^{i \theta_j}, e^{-i \theta_j}\}_{j=1}^m$ are the eigenvalues of $A$ on the unit circle, see \cite[Rem.~10.3.11, Rem.~12.2.4]{BCH}; if there are no such values, we set $m=0$ and the corresponding empty product to $1$. 

\begin{definition} \label{fluctang} We call $\{ \pi \} \cup \{\theta_j\}_{j=1}^m$ the \emph{fluctuation angles} of $\sigma$. 
\end{definition} 

The fact that $\vartheta<1$ implies that $B=\mathcal C(b_k)$. 
We first simplify the trigonometric factor $$T_k \defeq \prod_{j=1}^m (2 \sin(k \theta_j/2))^2  =  \prod_{j=1}^m (2-e^{ik \theta_j} - e^{-k\theta_j}) =  \sum_{\epsilon\,\in\,\{-1,0,1\}^m} c_\epsilon e^{ik\Theta_\epsilon},$$
where $\Theta_\epsilon = \sum_{i=1}^m \epsilon_i \theta_i$ and
$c_\epsilon = \prod_{i=1}^m d_{\epsilon_i}$ with $d_0 = 2$, $d_{\pm 1} = -1$.
It suffices to prove that for each $\Theta\in\R$
the Ces\`aro mean
$ L(\Theta) \defeq \mathcal C(\tilde b_k e^{ik\Theta})$
exists, where $\tilde b_k \defeq r_k \prod_{p \in S} |k|_p^{s_{p,k}} p^{-t_{p,k} |k|_p^{-1}}$; then the full Ces\`aro mean is
\begin{equation} \label{BasL} B = \sum_{\epsilon} c_\epsilon L(\Theta_\epsilon). \end{equation} 

To show this, we first regroup the remaining sum according to the $p$-adic absolute values of $k$, for $p \in S$. More precisely, for a tuple $\mathbf{j} = (j_p)_{p\in S} \in \Z_{\geq 0}^{|S|}$, set
$S_\mathbf{j} =\{ k \in \Z_{\geq 0} \colon v_p(k)=p^j, \forall p \in S\}$. Decompose
\[
L(\Theta) =  \sum_{\mathbf{j} \geq \mathbf 0}
  \frac{1}{X}\sum_{k\leq X}
  \mathbf 1_{k \in S_\mathbf{j}} \tilde b_k e^{ik\Theta}.
\]
Since $S_\mathbf{j}$ is a finite intersection of a finite union of arithmetic progressions, its characteristic function is periodic. Also, if $k \in S_\mathbf{j}$, then $|k|_p$ is fixed, so $\tilde b_k$ is also a periodic function in $k$ restricted to a fixed $S_{\mathbf{j}}$, by the periodicity of $r_k, s_{p,k}$ and $t_{p,k}$. This means that $\beta_\mathbf{j}(k) \defeq  \mathbf 1_{k \in S_\mathbf{j}} \tilde b_k$ is periodic in $k$, say, of period $P_\mathbf{j}$, so we can expand it in a discrete Fourier series for the group $\Z/P_{\mathbf j}\Z$, as follows: 
$$\beta_\mathbf{j}(k) = \sum_{n=0}^{P_{\mathbf j}-1} \hat{\beta}_\mathbf{j}(N)\,e^{2\pi ink/P_\mathbf{j}},$$ and then 
\[
L(\Theta) =  \sum_{\mathbf{j} \geq \mathbf 0} \sum_{n=0}^{P_{\mathbf{j}}-1} \hat\beta_\mathbf{j}(N) \mathcal C (e^{2 \pi i nk/P_{\mathbf{j}} + i k \Theta}).
\] 
For a fixed angle $\gamma$, we have the Ces\`aso limit $\mathcal C (e^{im\gamma})= \mathbf{1}_{\gamma \in 2\pi\Z}$ (summing the geometric series). Therefore
\[
L(\Theta) =  \sum_{\mathbf{j} \geq \mathbf 0} \sum_{n=0}^{P_{\mathbf{j}}-1} \hat\beta_\mathbf{j}(N) \mathbf 1_{\frac{n}{P_\mathbf{j}} + \frac{\Theta}{2 \pi} \in \Z}.
\] 
We need to establish the convergence of this sum. There is at most one value $n_\mathbf{j} \in \{0,\dots,P_\mathbf{j}-1\}$ for which $n_\mathbf{j}/P_\mathbf{j}+{\Theta}/{2 \pi} \in \Z$ (and the existence requires $\Theta \in \pi \Q$). Then, by definition of the discrete Fourier transform, we have 
\begin{equation} \label{LTheta} 
L(\Theta) =  \sum_{\mathbf{j} \geq \mathbf 0}  \hat\beta_\mathbf{j}(n_\mathbf{j}) = \sum_{\mathbf{j} \geq \mathbf 0} \frac{1}{P_\mathbf{j}} \sum_{a=0}^{P_\mathbf{j}-1} \tilde b_a \mathbf 1_{a \in S_{\mathbf j}} e^{-2 \pi i n_{\mathbf{j}} a/P_\mathbf{j}}.
\end{equation} 
Now from its definition, $\tilde b_a \leq C$ for the absolute constant $C=\max\{r_a \colon a = 0,\dots,\varpi_r\}$, where $\varpi_r$ is the period of $r_a$.  Thus, we have an upper bound
\[|L(\Theta)| \leq C \sum_{\mathbf{j} \geq \mathbf 0} \underbrace{\frac{1}{P_\mathbf{j}} \sum_{a=0}^{P_\mathbf{j}-1} \mathbf 1_{a \in S_{\mathbf j}}}_{=\delta_{\mathbf{j}}}.
\] 
Since $S_\mathbf{j}$ is periodic with period $P_\mathbf{j}$, $\delta_{\mathbf j}$ equals the density of the set $S_{\mathbf j}$, which is
$\delta_\mathbf{j} = \prod\limits_{p \in S} p^{-j} (1-1/p)$. Removing the finite factor $\prod_{p \in S} (1-1/p)$, we thus only need to establish the convergence of 
$$ \sum_{\mathbf{j} \geq \mathbf 0}  \prod_{p \in S} \frac{1}{p^{j_p}} = \prod_{p \in S} \sum_{j_p \geq 0} \frac{1}{p^{j_p}} = \prod_{p \in S}  \frac{1}{1-1/p},$$
which finishes the proof of the existence of the Ces\`aro limit.
\end{proof} 

\begin{remark} \label{nofluct}
If all $\sigma^k$ have at least one fixed point in $\mathcal{X}$ (as in the case, for example, for an endomorphism of a group, which has to fix the neutral element), then $A$ cannot have roots of unity as eigenvalues. This implies that none of the angles $\theta_j$ in \eqref{start0} can be rational multiples of $\pi$. If the set of flucuation angles $\{\theta_j\}_{j=1}^m \cup \{ \pi\}$ is rationally independent, then all the angles $\Theta_\epsilon$ that occur in the above proof have $L(\Theta_\epsilon)=0$, except for the value $\Theta_\epsilon=0$. In this case, $B=2^m \mathcal C(\tilde b_k)$. 
\end{remark}

\begin{remark} \label{Lambda1a} 
For a FAD-system with $\Lambda=1$, the fixed point count is periodic: $\sigma_k = r_k$, with period $\varpi$, see \cite[Prop.~10.3.6]{BCH}, and the dynamical zeta function is rational. Thus, the estimates \eqref{N1} and \eqref{M1} in Proposition \ref{Lambda1} hold. 
\end{remark}

\section{Examples of explicit asymptotic results} \label{exa}

In this section, we find the exact values of $B$ and $C$ in several cases of interest, more general than the examples presented in the introduction. We focus on the Ces\`aro limit $B$. For some cases, we also find an explicit formula for $C$, which requires knowing the entire asymptotics of $\zeta(u/\Lambda)$ near $u=1$ and using \eqref{c1} and \eqref{c2}. After that, we discuss the nature of the constants $B$ and $C$ (rationality, algebraicity, transcendence) and the relation of their nature to the geometry of the system, in case it arises from an algebraic group.  In Example (FF), it is clear that $B=C=1$, so we now deal with the other examples. If $S$ is a set of integers, we will use the suggestive notation $\mathbb{P}(n \colon n \in S)$ for the natural density of $S$ (if it exists).  

\subsection*{Multiplication on elliptic curves} In this subsection, we generalise the Example (E) to the following situation: suppose $E$ is an ordinary elliptic curve over $\overline \F_p$ with $p$ an odd prime, and suppose that $\sigma(P)=nP$ is multiplication by a positive integer $n$ on $E$. The number of fixed points of $\sigma^k$ is the number of elements of the torsion subgroup $E[n^k-1](\overline \F_p)$, and, decomposing into primary subgroups using that $E$ is ordinary, so $\# E[p^r](\overline \F_p) = p^r$, we find that $\sigma_k = (n^k-1)^2 \cdot |n^k-1|_p$. Thus, $\Lambda=n^2$ and $B=\mathcal C (|n^k-1|_p)$. 

If $p \mid n$, $|n^k-1|_p=1$ for all $k$, and $B=1$. We also find a rational zeta function $$\zeta(u/n^2) = \exp \sum_{k \geq 0} \frac{1}{k} \left( u^k - 2 (u/n)^k + (u/n^2)^k\right) = \frac{(1-u/n)^2}{(1-u)(1-u/n^2)} \underset{u \rightarrow 1^-}{\sim} (1-u)^{-1} \frac{n-1}{n+1}, $$ from which we infer that $C=(n/(n+1))^2$. 

In the more interesting case where $n$ is coprime to $p$, let $d$ denote the order of $n$ in $(\Z/p\Z)^*$ and $e \defeq v_p(n^d-1)$. The `lifting the exponent lemma' says that for all $k$ not divisible by $d$, $v_p(n^k-1)=0$, whereas if $d \mid k$, we have $v_p(n^k-1) = e + v_p(k/d)$. This implies 
$$ B =\mathbb{P}(d \nmid k) +  \lim_{X \rightarrow + \infty}  \frac{1}{X} \sum_{m \leq X/d} \frac{1}{p^{e+v_p(m)}} = 1 - \frac{1}{d} + \frac{1}{dp^e} \mathcal C(|m|_p).  $$
Now splitting summands by their exact $p$-valuations, we get $$\mathcal C (|m|_p) = \sum_{j \geq 0} \frac{1}{p^j} \mathbb P(m \colon p^j \mid m \wedge p^{j+1} \nmid m) = \sum_{j \geq 0} p^{-j} \left( \frac{1}{p^j} - \frac{1}{p^{j+1}} \right) = \frac{1-p^{-1}}{1-p^{-2}} = \frac{p}{p+1}. $$

To compute $C$, we use the exact asymptotics of $\zeta(u/n^2)$ near $u=1$. From the definition, we get $\log \zeta(u/n^2) = F(u)-2F(u/n)+F(u/n^2)$ where
\begin{align*} F(u) &\defeq \sum_{k \geq 1} |n^k-1|_p \frac{u^k}{k} =\sum_{\substack{k \geq 1 \\d \nmid k}} \frac{u^k}{k} + \frac{1}{dp^e} \sum_{m \geq 0} \frac{1}{p^{v_p(m)}} \frac{(u^d)^m}{m} \\ 
& =  \log \frac{\sqrt[d]{1-u^d}}{1-u} + \frac{1}{dp^e} \sum_{j \geq 0} \frac{1}{p^{2j}} \sum_{\substack{m \geq 1 \\ p \nmid m}} \frac{(u^{dp^j})^m}{m}  =  \log \Big(  \frac{\sqrt[d]{1-u^d}}{1-u} \Big(  \prod_{j \geq 0} \frac{(1-u^{dp^{j+1}})^{1/p^{2j+1}}}{ (1-u^{dp^{j}})^{1/p^{2j}}}  \Big)^{\frac{1}{dp^e}}    \Big)
\end{align*}
Combining factors in the product, this simplifies further to give 
\begin{equation} \label{expF} \exp F(u) = (1-u)^{-1} (1-u^d)^{\frac{1}{d}(1-\frac{1}{p^e})} \prod_{j \geq 1} (1-u^{dp^j})^{\frac{1}{p^{2j-1}} \cdot \frac{1}{dp^e}(1-\frac{1}{p})}. \end{equation}
To find the asymptotics near $u=1$, we divide out a factor $u-1$ for each factor in the product, notice that $(1-u^m)/(1-u) \sim m$ to get (with the sanity check that the exponent of $1-u$ is indeed $-B$), 
$$ \exp F(u) \sim (1-u)^{-B} \cdot d^{\frac{1}{d} ( 1-\frac{1}{p^{e-1}(p+1)} ) } \cdot p^{\frac{1}{dp^{e-2}(p-1)(p+1)^2}}. $$
On the other hand, $\exp F(1/n)$ and $\exp F(1/n^2)$ can be computed directly from \eqref{expF}, and if we collect all terms, we find the following. 

\begin{proposition} \label{BforE}
If $E$ is an ordinary elliptic curve over $\overline \F_p$ for an odd prime $p$, and $\sigma(P)=nP$ is multiplication-by-$n$ on $E$ for some integer $n$ coprime to $p$, then with $d$ the multiplicative order of $n$ modulo $p$ and $e=v_p(n^d-1)$, the constants in Theorem \ref{introthma} are 
\[B = 1-\frac{1}{d}\left( 1- \frac{p}{p^e(p+1)} \right) \in \mathbf Q  \]
and 
\[ C = d^{\frac{1}{d} ( 1-\frac{1}{p^{e-1}(p+1)} ) } \cdot p^{\frac{1}{dp^{e-2}(p-1)(p+1)^2}} \cdot \left(\frac{n^d+1}{n^d-1}\right)^{\frac{1}{d}(1-\frac{1}{p^{e-1}})} \cdot \left(\frac{n}{n+1}\right)^2 \cdot Q_p(n^d)^{\frac{1}{dp^e}(p-1)}    \]
where $$Q_p(x) \defeq \prod_{j \geq 0} \left(\frac{x^{p^j}+1}{x^{p^j}-1}\right)^{\frac{1}{p^{2j}}};$$
in particular, $B \neq 1$. On the other hand, if $n$ is divisible by $p$, $B=1$ and $C=n^2/(n+1)^2$.  \qed
\end{proposition} 

Example (E) from the introduction is the case $p=3$, $n=2$, $d=2$, $e=1$. 

\begin{remark} The case where $E$ is supersingular is similar, with $\sigma_k = (n^k-1)^2 \cdot |n^k-1|^2_p$. In characteristic $p=2$, lifting the exponent works a bit differently. It may be interesting to generalise further to an arbitrary $\sigma \in \End(E)$ not given by an integer (recall that $\End(E)$ is an imaginary quadratic order in the ordinary case, and a maximal order in a quaternion algebra in the supersingular case). 
\end{remark}

\begin{remark} We conjecture that $C$ is transcendental for $n$ not divisible by $p$, which hinges on the nature of the function $Q_p(x)$. 
The function $Q_p(x)$ relates to the zeta function by Mellin transform $\mathcal M$, as follows: 
\begin{align*} \mathcal M [ \log Q_p(e^{-t}) ] (s)  & = 2 \Gamma(s) \sum_{j\geq 0} \frac{1}{p^{2j}} \sum_{\ell \geq 0} \frac{1}{2\ell+1} \cdot \frac{1}{(p^j(2\ell+1))^s} 
= \Gamma(s) \frac{1-2^{-s-1}}{1-p^{-s-2}} \zeta(s+1).
\end{align*}
The function $Q_p(x)$ is a Mahler function, in that it satisfies a functional equation of the form 
$$ Q_p(x^p) = Q_p(x)^{p^2} \cdot \left( \frac{x-1}{x+1}\right)^{p^2}.$$
By \cite[Thm.~1.3(i)]{Nishioka}, if the function $Q_p(x)$ were algebraic over $\mathbf C(x)$, then it would in fact be rational, i.e., $Q_p(x) \in \mathbf C(x)$, which is impossible for degree reasons (with $p>1$). The functional equation, applied to $Q(1/x)$, is just outside the scope of Nishioka's \cite{Nishioka} method to establish transcendence at rational values, so we leave it as an open problem.  
\end{remark}

\subsection*{Additive cellular automata and vector group endomorphisms} For a confined endomorphism $\sigma$ of $\Ga^r$, we have $\sigma_k = p^{kc-t_k|k|_p^{-1}}$ for non-negative integers $c$ and $t_k$, where $t_k$ is periodic of period $\varpi$ coprime to $p$, see \cite[Prop.~G]{BCca}. Here, $\Lambda=p^c$, and, writing $\varpi$ for the period of $t_k$, and using that this period is coprime to $p$, we find 
\begin{align} \label{BforCA}
B & =   \lim_{X \rightarrow \infty} \frac{1}{X} \sum_{a=0}^{\varpi-1} \sum_{j \geq 0} \sum_{\substack{k \leq X \\ v_p(k)=j \\ k \equiv a \text{ mod } \varpi}} \frac{1}{p^{t_a p^j}} %\nonumber \\& 
=  \lim_{X \rightarrow \infty}  \sum_{a=0}^{\varpi-1} \sum_{j \geq 0} \underbrace{\mathbb{P}(k \colon v_p(k)=j \wedge k \equiv a \text{ mod } \varpi)}_{\frac{1}{\varphi(\varpi)} p^{-j} \left( 1 - \frac{1}{p} \right)} \frac{1}{p^{t_a p^j}} \nonumber \\ 
& = \frac{p-1}{\varphi(\varpi)} \sum_{a=0}^{\varpi-1} \mathcal C (p^{-1-j-t_a p^j}). 
\end{align} 
This number is similar to the so-called `Kempner number' $\mathcal C(1/2^{2^j})$ (see, e.g., \cite{AdamKemp}) and we can prove that $B$ is transcendental if and only if there is a $t_k \neq 0$, using the method of Mahler (e.g., from Nishioka \cite{Nishioka}), as follows. Define the function $$F(x) \defeq  \sum_{a=0}^{\varpi-1} \sum_{j \geq 0} \frac{x^{t_a p^j}}{p^j},$$
and notice that $F$ satisfies a functional equation in Mahler form 
$$ \frac{1}{p} F(x^p) = F(x) - \sum_{a=0}^{\varpi-1} x^{t_a}. $$ 
We can conclude that $B=F(1/p)$ is transcendental, by checking the following conditions required by \cite[Thm.~1.5.1]{Nishioka}: 
\begin{enumerate} 
\item $1/q< \min \{1,R\}$, where $R$ is the convergence radius of $F(x)$ as power series, which is true since $R=1$; 
\item the degree $m$ of the Mahler equation in $F(x)$ is $1$, the exponent $d$ in $F(x^p)$ is $p$, and the exponent of $F(x^p)$ in the equation is $n=1$, and we require $\max\{d,m\} n^2 < d^2$, which holds; 
\item finding suitable functions $g_0(x,u)$ and $g_1(x,u)$ such that  $g(x) = g_0(x,u)/p + g_1(x,u) (-u+\sum x^{t_a})$ is independent of $u$, namely, $g_0=u$ and $g_1=1/p$; then $g(x) = \frac{1}{p} \sum_{a=0}^{\varpi-1} x^{t_a}$, we require $g(1/p^d) \neq 0$ for all $d$, which holds since there exists a least one $a$ with $t_a\neq 0$;
\item the function $F(x)$ is not algebraic over $\mathbf C(x)$, which holds, since, if it were algebraic, it would be rational by \cite[Thm.~1.3(i)]{Nishioka}; but then the functional equation implies that, if it is rational of degree $d$, $dp=d-\max\{t_a\}$, and this happens precisely if all $t_a$ are zero.  
\end{enumerate}
Finally, if all $t_a$ are zero, summing the geometric series in \eqref{BforCA} shows that $B=1$. 

As we have remarked before, such $\sigma$ defined over $\F_p$ correspond bijectively to the class of additive cellular automata $\tau$ with alphabet $\F_p^r$. The fixed point count for iterates of $\tau$ is identical to that of the corresponding $\F_p$-rational endomorphism of the vector group $\Ga^r$ over $\F_p$ \cite{BCH}. Hence the result also applies to this class of examples. 

We conclude the following: 
\begin{proposition} \label{gatransc}
If $\sigma$ is a confined endomorphism of a vector group $\Ga^r$ over $\overline \F_p$ (in particular, if $\sigma$ is an additive cellular automaton on the alphabet $\F_p^r$) then \[ B = \frac{p-1}{\varphi(\varpi)} \sum_{a=0}^{\varpi-1} \mathcal C (p^{-1-j-t_a p^j})\] is transcendental, or $B=1$, where the latter case happens precisely if $\zeta_\sigma(z) \in \mathbf C(z)$ is a rational function, or, equivalently, if all $t_a=0$. \qed
\end{proposition} 

\begin{remark} 
Since $B$ has a very sparse $p$-ary expansion, one may also use Ridout's theorem \cite{Ridout} to prove transcendence. 
\end{remark} 

In the specific example (GA) from the introduction, we have $p=2$, $c=1$, $t_n=1$ for all $n$, and so $B=\mathcal C (p^{-1-j-p^j}) = \sum_{j \geq 0} 2^{-1-j-2^j}$, as indicated. 
To find $C$ in this case, we again revert to finding the precise asymptotics of $\zeta(u/2)$ around $u=1$. Substituting the value of $\sigma_k$ and writing $k=2^j m$ with $m$ odd, we find 
$$\log \zeta(u/2) = \sum_{k \geq 1} \frac{1}{2^{2^{v_2(k)}}} \frac{u^k}{k} = \sum_{j \geq 0} \frac{1}{2^{2^j+j}} \sum_{m \text{ odd } \geq 1} \frac{(u^{2^j})^m}{m}.$$
The inner sum is $1/2 \log((1+u^{2^j})/(1-u^{2^j}))$, and hence for $u \rightarrow 1^-$, we find 
$$ \zeta(u/2) = \prod_{j \geq 0} \Big( \frac{1+u^{2^j}}{1-u^{2^j}} \Big)^{1/2^{2^j+j+1}} \sim (1-u)^{-B} \prod_{j \geq 0} (2/2^j)^{1/2^{2^j+j+1}} \sim (1-u)^{-B} 2^{B-A}$$
with $A = \sum_{j \geq 0} j/2^{2^j+j+1}$. As in \eqref{c2}, with $\Lambda=2$, we thus find $C=2^{1+B-A}.$

\subsection*{A torus endomorphism without unique dominant root}  
As final example, we consider the torus endomorphism given in Example (GM) from the introduction. This is an example where in Formula \eqref{start0}, $m \neq 0$, so that there is a sinusoidal fluctuation in the fix point count sequence $\sigma_k$; this is the case where the linear recurrence for the factor $\det(A^n-1)$ in the expression for $\sigma_k$ does not have a unique so-called \emph{dominant root}, i.e., $A$ has several eigenvalues of maximal absolute value ($=\Lambda/c$).  

In the Example (GM) at hand, 
$$ \sigma_k = 4\Lambda^k \sin^2(k\theta/2) r_k |k|_5^{s_k} \mbox{ with } (r_k,s_k) = \left\{ \begin{array}{ll} (5^{-2},4) & \mbox{ if } 3 \mid n; \\ (1,0) & \mbox{ otherwise }, \end{array} \right. $$ 
where $\Lambda = \frac{1}{4}
   \left(3+\sqrt{5} +  \sqrt{6 \sqrt{5}-2}\right) \approx 2.15372$ and $8 \cos \theta = 3-\sqrt{5}$ (so $\theta \approx 1.37863\dots$), see \cite[Ex.~3.5.2(ii), Ex.~10.3.12]{BCH}.   
Now $\theta$ is not a rational multiple of $\pi$: if this were the case, $2 \cos \theta$ would be an algebraic integer, but it has minimal polynomial $4x^2-6x+1$. Hence by Remark \ref{nofluct}, $B= 2 \mathcal C(r_k |k|_5^{s_k}).$ This can be calculated by methods similar to the ones above: 
\begin{align*}
B &= \frac{2}{X} \Big( \sum_{\substack{k \leq X \\ 3 \nmid k}} 1 + \sum_{\substack{k \leq X \\ 3 \mid k}} \frac{1}{5^2} |k|_5^4 \Big)  
= \frac{4}{3} + 2 \sum_{j \geq 0} \underbrace{\mathbb{P}(3 \mid k \wedge v_5(k)=j)}_{1/3 \cdot (1/5^j-1/5^{j+1})}  \frac{1}{5^2} \frac{1}{5^{4j}} \\
& = \frac{4}{3} + \frac{8}{3\cdot 5^3} \frac{1}{1-1/5^5} = \frac{2 \cdot 23^2}{11 \cdot 71}. 
\end{align*}

We refrain from computing the constant $C$ in the asymptotic formula \eqref{N} for this case, but it can be computed with some effort from the explicit representation if $\sigma_k$ is given as above. 

\begin{remark} In this case, $B>1$, whereas for the other examples (E) and (GA), $B<1$ (and recall that for (FF) we have $B=1$). This shows that there can be a polynomial correction that makes $N(X)$ asymptotically larger or smaller than the purely exponential behaviour of type $C \Lambda^X$. 
\end{remark} 

\subsection*{The Ces\`aro mean and the geometry of the algebraic group} In this subsection, we prove Theorem \ref{dich} which relates the geometry of a general algebraic group $G$, more specifically its endoseries, to the possible rationality/algebraicity/transcendence of the corresponding Ces\`aro mean $B$. 

\begin{proof}[Proof of Theorem \ref{dich}] 
Statement (iii) was already shown in Proposition \ref{gatransc}. 
Thus, we consider an algebraic group $G$ over $\overline \F_p$ with an endomorphism $\sigma$ such that the zeta function of the induced action of $\sigma$ on every vector group quotient occurring in an endoseries for $G$ is rational. This is a FAD-system with $S=\{p\}$, so we leave out $S$ and $p$ from the notation, writing $s_{p,k}=s_k$ etc. Recalling from \cite[\S 8.2]{BCH} that in this case, all $t_k=0$ on such vector group quotients, and by multiplicativity of the fixed point count over short exact sequences of algebraic groups \cite[Prop.~4.8.1]{BCH} the same holds for the total count. 
We now take a look at the proof of Theorem \ref{FADisgood} for this specific case, and find from \eqref{BasL} that we only need to prove that the various occurring $L(\Theta)$ are algebraic. We only need to consider the case where $\Theta \in \pi \Q$, say, $\Theta=2 \pi q$ for $q \in \Q$ (since otherwise, $L(\Theta)=0$, as follows from the proof of Theorem \ref{FADisgood}). Assuming this, we look at formula \eqref{LTheta}, where now $\mathbf j$ is just a scalar $j$, and the period $P_\mathbf{j}$ can be taken as the common period $\varpi$ of $r_k$ and $s_k$. We find 
\begin{equation} \label{LThetaalg}  L(\Theta) = \frac{1}{\varpi}  \sum_{a=0}^{\varpi-1} r_a e^{2 \pi q a} \sum_{j \geq 0} p^{-j s_a} = \frac{1}{\varpi}  \sum_{a=0}^{\varpi-1} \frac{r_a}{1-p^{-s_a}}  e^{2 \pi q a} \in \Q^{\mathrm{ab}},
\end{equation} 
which proves the claim in (i). Finally, for (ii), notice that if the fluctuation angles are linearly independent over $\Q$, then by Remark \ref{nofluct}, we have $B=2^m L(0)$ with 
$L(0) \in \Q$, since this is the case where $q=0$ in Formula \eqref{LThetaalg}. 
\end{proof} 

\vspace*{-4mm}
\bibliographystyle{amsplain}
\providecommand{\href}[2]{#2}

\end{document}